\documentclass[12pt]{amsart}
\usepackage{epsfig} 
\usepackage{color}
\usepackage{psfrag}
\usepackage{graphicx}
\usepackage{amssymb}
\usepackage{amsmath}
\usepackage{latexsym}
\usepackage{enumerate}
\usepackage[mathscr]{eucal}
\usepackage[isolatin]{inputenc}
\usepackage{upref}

\makeatletter
\@namedef{subjclassname@2010}{%
 \textup{2010} Mathematics Subject Classification}
\makeatother

\newtheorem{thm}{Theorem}[section]

\newtheorem{lem}[thm]{Lemma}

\theoremstyle{definition}
\newtheorem{defin}[thm]{Definition}
\newtheorem{rem}[thm]{Remark}

\numberwithin{equation}{section}

 \newcommand{\setN}{\mathbb{N}}
 
 \newcommand{\setQ}{\mathbb{Q}}
 \newcommand{\setR}{\mathbb{R}}
 
 \newcommand{\setZ}{\mathbb{Z}}

\begin{document}

\baselineskip=17pt

\title[]
{Reduced set theory}

\author{Matthias Kunik}
\address{Universit\"{a}t Magdeburg\\
IAN \\
Geb\"{a}ude 02 \\
Universit\"{a}tsplatz 2 \\
D-39106 Magdeburg \\
Germany}
\email{matthias.kunik@ovgu.de}

\date{\today}
\maketitle


\begin{abstract}
We present a new fragment of axiomatic set theory for pure sets 
and for the iteration of power sets within given transitive sets. It turns out that this formal system admits an interesting hierarchy of models with true membership relation and with only finite or countably infinite ordinals. Still a considerable part of mathematics can be formalized within this system.
\end{abstract}

{\bf Keywords:} Formal mathematical systems, axiomatic set theory.\\

Mathematics Subject Classification: 03F03, 03E30\\

\section{Introduction}\label{intro}
%
%
In this article we present a generalization 
of Zermelo-Fraenkel set theory (ZFC), 
starting with a fragment of axiomatic set theory
which we will call RST, for reduced set theory.
We are only dealing with sets whose members are sets again.
First we will list the principles how we are dealing with sets in RST without using a formal language. These principles will be given
a precise form as axioms in Section \ref{RST}, where we use the formal systems from \cite[Sections 3,4]{Ku}.
For this purpose we need some crucial set constructions given
in Section \ref{construct}. 

A set $U$ is called \textit{transitive} iff $Y \subseteq U$ 
for all $Y \in U$. Every set can be extended to a transitive set
by Theorem \ref{tcthm1}. 
By $\mathcal{P}(Y) =\{ V : V \subseteq Y\}$
we denote the power set of $Y$.
We say a set $U$ is \textit{subset-friendly} iff 
\begin{itemize}
\item[1.] $\emptyset \in U$\,.
\item[2.] $U$ is transitive\,.
\item[3.] For all $Y \in U$ we have $\mathcal{P}(Y) \in U$\,.
\item[4.] For all $Y, Z \in U$ we have a transitive set\\ 
$V \in U$ with $\{Y,Z\} \subseteq V$\,.
\end{itemize}

\noindent
Now we are listing six principles according to which we are dealing with sets in RST. For sets $A$, $B$, $U$, $V$, $Y$ these are given by
\begin{itemize}
\item[P1.] \textit{Principle of extensionality.}
If $A$ and $B$ have the same elements, then $A=B$.
\item[P2.] \textit{Subset principle.}
If $\mathcal{F}$ is a property which may depend on previously given sets, then we can form the subset of $A$ given by
$U = \{Y :\,\mbox{there~holds~} Y \in A 
\mbox{~and~} Y \mbox{~has~property~}\mathcal{F}\}\,.$\\
Especially the empty set $\emptyset$ can be obtained from this principle.
\item[P3.] \textit{Principle of regularity.}
If $U$ is not the empty set,
then we have $Y \in U$ with $U \cap Y=\emptyset$.
\item[P4.] \textit{Principle for pairing of sets.}
If $A$ and $B$ are given, then we can find a set $U$
with $A \in U$ and $B \in U$. We can combine this with (P2)
to form $U=\{A, B\}$.
\item[P5.] \textit{Principle for subset-friendly sets.}
If $A$ is given, then we have a subset-friendly set $U$
with $A \in U$.
\item[P6.] \textit{Principle of choice.}
If $U$ has only nonempty and pairwise disjoint elements
then we can find a set $Y$ with the following property:
For every member $A \in U$ there exists exactly one set $V$ with
$Y \cap A = \{V\}$.
\end{itemize}
(P1)-(P6) will be given an exact form in Section \ref{RST} with corresponding RST-axioms (A1)-(A6).
In Section \ref{construct} we will give a motivation
for these principles.
The novel feature of (P5) is that it contains the set $A$ as 
para\-meter. Hence we can use it step by step.
We will first provide a subset-friendly set $U$ with $A=\emptyset \in U$.
Then we can apply (P5) to $A=U$ again, and so on.
The correctness of (P5) is guaranteed by Theorem \ref{spthm1}.

To specify the property $\mathcal{F}$ in (P2) exactly, we use the formal language of the predicate calculus. The language of set theory 
in Section \ref{RST} consists of a set
$X=\{\,{\bf x_1}\,,\,{\bf x_2}\,,\,{\bf x_3}\,,\, \ldots \,\}$
of variables, the equality predicate $\sim$ and a binary predicate
$\in$ for membership relation. Using variables $x,y \in X$
we start with atomic formulas $\sim x,y$ and $\in x,y$.
Let formulas $F$, $G$ be constructed previously.
Then we can form step by step the connectives 
$$
\neg F\,, \quad \to F G\,, \quad \& F G\,, \quad \vee F G\,, \quad
\leftrightarrow F G
$$
and the formulas
$$
\forall x\, F\,, ~ \exists x\,F\,.
$$

The Hilbert-style predicate calculus we use is only slightly different
from that of Shoenfield in his textbook \cite{Shoenfield}. 
In Section \ref{FMS} we provide results from \cite{Shoenfield} concerning conservative extensions of formal mathematical systems. 
This is applied in Section \ref{RST} to enrich the formal language of RST
with new symbols like $\emptyset$, $\subseteq$, $\cup$, $\mathcal{P}$,
without changing the provability of the original formulas in RST.
We present this result in Theorem \ref{nice_axioms}. Extensions of RST which we will call subset-friendly theories are naturally included in our approach,
see Definition \ref{sftheories}. In these theories the subset axioms 
still remain valid with new symbols in the formulas.
Theorem \ref{nice_axioms} also says that all axioms for sets in ZFC
given in \cite[Chapter 9]{Shoenfield} are already provable in RST, apart from the replacement axioms. In Section \ref{RST} we will 
use a semi formal approach to ensure these results, i.e.\ we will only employ well known results from elementary proof theory for the formal system RST and its extensions. The key idea behind the axioms (A5) 
for subset-friendly sets is that we can use them iteratively. 
In this way we have a sufficiently large set as background available. 
Within this set we can now perform the operations
listed in Remark \ref{sfrich}. Then we apply the subset axioms directly instead of the replacement axioms.

In Section \ref{RST_Model} we study a hierarchy of 
models for RST. The universe of each model is a 
subset-friendly set $\mathcal{U}_n$ given in \eqref{unisets}, and the membership relation in each model is the true membership relation between the individuals in the universe $\mathcal{U}_n$.
This is shown in Theorem \ref{thm_models}.
These are only the simplest models. All of them have only finite or countably infinite ordinals, see Theorem \ref{finalthm}. 
We will see that a considerable part of mathematics
can be formalized within RST. 
On the other hand, the models for RST without uncountable ordinals 
from Theorem \ref{thm_models} 
violate Zermelo's well-ordering theorem.
Zermelo's well-ordering theorem states 
that for every set $A$ there is 
a bijective mapping from an ordinal to $A$.
Nevertheless, because of (P6) any set can be well-ordered in RST.

At this place it is instructive to compare our approach with the
study of a so called Zermelo universe given in
Moschovakis \cite{Mosch}.
A Zermelo universe is a special model of a seminal axiomatic set theory originally given by Zermelo without using replacement axioms.
In our context with pure sets let $\omega$ be the set of all finite ordinals and let $U$ be a transitive set. 
We say that $U$ is a Zermelo universe if
$\omega \in U$ and if $U$ is closed under the opera\-tions 
of pairing, union and power sets. 
This looks similar to our definition of subset-friendly sets $U$, apart from the fact that for our application of (P5) we can already start with 
$A = \emptyset$. Based on the formal theory ZFC for pure sets
which makes use of the replacement axioms \eqref{repax}, 
the corresponding set constructions given in \eqref{tcsp}, Theorem \ref{tcthm1}, Theorem \ref{spthm1}
are indeed the same as in \cite[11.10]{Mosch}
and \cite[11.15]{Mosch}, with only a slightly different notation.
For these set constructions of the form
$\bigcup_{n=0}^{\infty} A_n$
we use the replacement axioms
only impli\-citly, without mentioning them.

However, in RST we do not make use of the replacement axioms.
Hence we really need the new Definition
\ref{interpret_axioms}(b) of subset-friendly sets 
for the axioms of RST in its given form. 
Our definition guarantees that any given subset-friendly set has the desired closure properties and sufficiently many transitive sets available as elements. Another difference in our approach is this: 
In our context the subset-friendly sets obtained from (P5) are members
of a larger universe for RST, but in general not universes of RST-models.
We use the set operation $\mathcal{SP}(\,\cdot\,)$ 
from \eqref{tcsp} and illustrate this for the construction of the simplest model set ${\mathcal U}_0$ of RST in \eqref{unisets}, where the subset-friendly sets
\begin{align*}
V_{0,0} \in V_{0,1} \in V_{0,2} \ldots
\end{align*}
are only special individuals of its universe
$$
{\mathcal U}_0 = \bigcup_{k=0}^{\infty} V_{0,k}\,.
$$
Then $A=\emptyset$, $U=\mathcal{SP}(\emptyset)=V_{0,0}$ and
$A=V_{0,k-1}$, $U=\mathcal{SP}(V_{0,k-1}) = V_{0,k}$ for $k \in \setN$
satisfy (P5), respectively. Finally, Theorem \ref{finalthm2} 
also guarantees an interesting countable transitive model for RST.


\section{Some notations and constructions with pure sets}\label{construct}
Throughout the whole paper we will only consider \textit{pure sets} without urelements. In addition we require
for every set $A=A_0$ that it does not allow an 
infinite descending sequence of sets
\begin{equation}\label{forbidden_sequence}
\ldots \in A_3 \in A_2 \in A_1 \in A_0\,.
\end{equation}
This is guaranteed by the principle of regularity which states that for every set $U \neq \emptyset$ there exists a set $Y$ with $Y \in U$ and $U \cap Y = \emptyset$.\\
Assume that we have an infinite sequence $A_0,A_1,A_2,\ldots$
satisfying \eqref{forbidden_sequence} and form the set
$U=\{A_0,A_1,A_2,\ldots\}$. 
If we apply the principle of regularity to $U$, 
then we can choose $Y=A_k$ with $k \in \setN_0$
and $U \cap Y=\emptyset$ and obtain the contradiction
$A_{k+1} \in A_k=Y$, $A_{k+1} \in U \cap Y$.\\
On the other hand we prescribe a nonempty set $U$ and choose $Y_0 \in U$.
For $U \cap Y_0 = \emptyset$ the set $U$ satisfies the principle of regularity with $Y=Y_0$.
Otherwise we have $Y_1 \in U \cap Y_0$, and if $Y_0,\ldots,Y_j$ with $j \in \setN$ 
are already given such that $Y_{k} \in U \cap Y_{k-1}$ for all $k=1, \ldots ,j$, then
$\begin{displaystyle}
Y_j \in Y_{j-1} \in \ldots  \in Y_0 \,,
\end{displaystyle}$
and we can proceed to form $Y_{j+1} \in U \cap Y_j$ as long as $U \cap Y_j \neq \emptyset$. 
Since we have excluded an infinite descending sequence of sets in \eqref{forbidden_sequence}, 
we obtain an index $n \in \setN_0$ with $Y_n \in U$ and $U \cap Y_n = \emptyset$.
We see that the principle of regularity (P3) given in the introduction is
a simple and natural condition.

The union of a set $A$ is
$$
\cup(A)=\{C : \textrm{there exists } B \in A \textrm{ with } 
C \in B \,\}\,,$$
and its power set is
$\begin{displaystyle}
\mathcal{P}(A)=\{B : B \subseteq A \}\,.
\end{displaystyle}$
More generally we have
\begin{equation*}
\cup^0(A) = A ~~\,\textrm{and}~~ 
\cup^{n}(A)=\cup(\cup^{n-1}(A))
~~\,\textrm{for all } n \in \setN
\end{equation*}
as well as
\begin{equation*}
\mathcal{P}^0(A) = A ~~\,\textrm{and }~~ 
\mathcal{P}^{n}(A)=\mathcal{P}(\mathcal{P}^{n-1}(A))
~~\,\textrm{for all } n \in \setN \,.
\end{equation*}

For a given set $A$ we also define 
\begin{equation}\label{tcsp}
\mathcal{TC}(A)=\bigcup_{n=0}^{\infty} \cup^{n}(A) 
\quad  \textrm{and} \quad
\mathcal{SP}(A)=\bigcup_{n=0}^{\infty} \mathcal{P}^{n}(A)\,.
\end{equation}
If $A \neq \emptyset$ is a set or a proper class of sets, then we have the set
$$
\cap(A)=\{ C : C \in B ~~\,\textrm{for all } ~ B \in A \}\,.
$$

\begin{defin}\label{interpret_axioms}
The following two definitions are crucial for the interpretation
of new set axioms which will be introduced in Section \ref{RST}.
\begin{itemize}
\item[(a)] A set $U$ is called \textit{transitive} iff $Y \subseteq U$ 
for all $Y \in U$.
\item[(b)] A set $U$ is called \textit{subset-friendly} iff 
\begin{itemize}
\item[1.] $\emptyset \in U$\,.
\item[2.] $U$ is transitive\,.
\item[3.] For all $Y \in U$ we have $\mathcal{P}(Y) \in U$\,.
\item[4.] For all $Y, Z \in U$ we have a transitive set 
$V \in U$ with $\{Y,Z\} \subseteq V$\,.
\end{itemize}
\end{itemize}
\end{defin}
\begin{thm}\label{tcthm1}
The following statements are equivalent for any set $T$.
\begin{itemize}
\item[(a)] $T$ is transitive\,,
\item[(b)] $T \subseteq \mathcal{P}(T)$\,,
\item[(c)] $\cup(T) \subseteq T$\,,
\item[(d)] $T = \mathcal{TC}(T)$\,.
\end{itemize}

For every set $A$ the so called transitive closure $\mathcal{TC}(A)$ of $A$
is a transitive set such that $A \subseteq \mathcal{TC}(A)$, the smallest transitive set $T$ with $A \subseteq T$, i.e.\ 
$
\mathcal{TC}(A) = \cap\,(\,\{ T \,:\,A \subseteq T \textrm{ and } T 
\textrm{ is a transitive set}\,\}\,)\,.
$
\end{thm}
\begin{proof}
 $T$ is transitive $\iff$ ($B \in T \implies B \subseteq T$ 
for all sets $B$)\\
$\iff$ ($B \in T \implies B \in \mathcal{P}(T)$ 
for all sets $B$) $\iff$ $T \subseteq \mathcal{P}(T)$.\\
Hence (a) and (b) are equivalent.
We have for all sets $A$ and $T$:\\
$A \in \cup(\mathcal{P}(T))$ $\iff$ 
(there exists $B \in \mathcal{P}(T)$ with $A \in B$)\\
$\iff$ (there exists $B \subseteq T$ with $A \in B$)
$\iff A \in T$\,.\\
We obtain $\cup(\mathcal{P}(T))=T$ for all sets $T$.
Especially for transitive $T$ we can use (b) and 
conclude that $\cup(T) \subseteq T$.
Now let $\cup(T) \subseteq T$ and assume that $T$
is not transitive. 
Then there exists $A$ with $A \in T$ which violates $A \subseteq T$.
Hence there exists $B \in A$ with $B \notin T$.
Here we obtain the contradiction $B \in \cup(T)\subseteq T$.
We see that (a)-(c) are equivalent.
The condition $A \subseteq \mathcal{TC}(A)$ is clear from the definition of
$\mathcal{TC}(A)$. In order to prove that $\mathcal{TC}(A)$ is transitive
we assume that $\cup(\mathcal{TC}(A))$ 
is not a subset of $\mathcal{TC}(A)$.
Then we have $C \in \cup(\mathcal{TC}(A))$ with
$C \notin \mathcal{TC}(A)$, and there exists $B \in \mathcal{TC}(A)$
with $C \in B$ and $B \in \cup^n(A)$ for some $n \in \setN_0$. 
We obtain the contradiction
$C \in \cup(\cup^n(A))=\cup^{n+1}(A) \subseteq \mathcal{TC}(A)$
and conclude that $\mathcal{TC}(A)$ is transitive.
Finally, if $A \subseteq T$ for a transitive set $T$,
then $\mathcal{TC}(A) \subseteq \mathcal{TC}(T)$
and $\mathcal{TC}(T)=T$ from (c).
\end{proof}
\begin{thm}\label{tcthm2}
For all sets $A, B$ we have
\begin{itemize} 
\item[(a)] $A \in B \implies A \neq B$,
\item[(b)] $A \notin \mathcal{TC}(A)$.
\end{itemize}
\end{thm}
\begin{proof}
(a) Assume $A \in B$ and $A=B$. Then $A \in A$ contradicts
the principle of regularity, if applied to $U = \{A\}$\,.
(b) We see by induction over $n \in \setN_0$ that
$B \in \cup^{n}(A)$ iff for all $k \in \{0,\ldots,n\}$
there are sets $A_k$ with $A_k \in \cup^{k}(\mathcal{P}(A))$ 
and $B \in A_{n} \in \ldots \in A_{0}=A$\,.\\
Assume that $A \in \mathcal{TC}(A)$.
Then $A \in \cup^{n}(A)$ for some $n \in \setN_0$,
and from $A \in A_{n} \in \ldots \in A_{0}=A$
with sets $A_k \in \cup^{k}(\mathcal{P}(A))$ 
for $k \in \{0,\ldots,n\}$
we obtain a forbidden periodic sequence of the form given in
\eqref{forbidden_sequence}.
\end{proof}

\begin{defin}\label{setmap}
For any set $A$ we define 
\begin{itemize}
\item[(a)] its successor $A^+=\{A\} \cup A$\,,
\item[(b)] the transitive set
$\begin{displaystyle}
\mathcal{S}_{+}(A)= \mathcal{TC}(\{A\})=
\{A\} \cup \mathcal{TC}(A)\,.
\end{displaystyle}$
Let $T$ be any further set and put 
$\begin{displaystyle}
\mathcal{S}_{-}(T)= \cup(T \setminus \cup(T))\,
\end{displaystyle}$\,.
\noindent
Then we see from Theorem \ref{tcthm2} that 
$\mathcal{S}_{-}(\mathcal{S}_{+}(A))= A$.
Note that $A^+=\mathcal{S}_{+}(A)$
if $A$ is transitive, see Theorem \ref{tcthm1}.
\item[(c)] With a further set $B$ let 
$\langle A,B \rangle = \{ \{A\}, \{A,B\}\}$ 
be the ordered pair of $A$ and $B$.
\end{itemize}
\end{defin}

\begin{thm}\label{spthm1}
Let $A$ be a set and $T=\mathcal{TC}(A)$. Then
$\mathcal{SP}(T)$ is a subset-friendly set 
with $A \in \mathcal{SP}(T)$. 
Now $\mathcal{SP}(T)$  is the smallest subset-friendly set $U$ with $A \in U$, i.e.
$$
\mathcal{SP}(T) 
= \cap \, (\,\{ U \,:\,A \in U \textrm{ and } U 
\textrm{ is a subset-friendly set}\,\}\,)\,.
$$
\end{thm}
\begin{proof} 
We make use of \eqref{tcsp} and Theorem \ref{tcthm1}.
We see $A \subseteq T=\mathcal{TC}(A)$
and hence $A \in \mathcal{P}(T)$, $A \in \mathcal{SP}(T)$. 
Next we will check that $\mathcal{SP}(T)$
satisfies the four properties for a subset-friendly set.  
\begin{itemize}
\item[1.]
If $U \neq \emptyset$ is a transitive set then we have a set
$B \in U$ with $B \cap U = \emptyset$ from the 
principle of regularity. 
We have $B = \emptyset$ since $C \in B$ implies
$C \in U$ from the transitivity of $U$ and 
the contradiction $C \in B \cap U$. Therefore
the conditions $U \neq \emptyset$ and $\emptyset \in U$
are equivalent for each transitive set $U$, especially for $U$ in Definition \ref{interpret_axioms}(b).
\item[2.]
We see from Theorem \ref{tcthm1} by complete induction 
that $\mathcal{P}^n(T)$ is a transitive set for all $n \in \setN_0$
with $\mathcal{P}^n(T) \subseteq \mathcal{P}^{n+1}(T)$. 
We conclude that $\mathcal{SP}(T)$ is a transitive set
with $\emptyset \in \mathcal{SP}(T)$. 
\item[3.] Let $Y \in \mathcal{SP}(T)$.
Then $Y \in \mathcal{P}^n(T)$ for some $n \in \setN_0$
and hence $Y \subseteq \mathcal{P}^n(T)$
as well as $\mathcal{P}(Y) \subseteq \mathcal{P}^{n+1}(T)$
from the transitivity of $\mathcal{P}^n(T)$.
We obtain that 
$\mathcal{P}(Y) \in \mathcal{P}^{n+2}(T) \subseteq \mathcal{SP}(T)$.
\item[4.] For $Y, Z \in \mathcal{SP}(T)$ we have
$Y, Z \in \mathcal{P}^j(T)$
and 
$\{Y, Z\} \subseteq \mathcal{P}^{j}(T)$
for $j \in \setN_0$ large enough with the transitive set
$\mathcal{P}^{j}(T) \in \mathcal{SP}(T)$.
\end{itemize}

We see that $\mathcal{SP}(T)$ is a subset-friendly set
with $A \in \mathcal{SP}(T)$.
Now let $U$ be a subset-friendly set with $A \in U$.
We recall the function $\mathcal{S}_+$ in Definition \ref{setmap}.
Let $V \in U$ be transitive with $\{A\} \subseteq V$, i.e.\
$A \in V$, see Condition 4 in 
Definition \ref{interpret_axioms}(b) with $Y=Z=A$.
Then we must have $\mathcal{S}_+(A)=\{A\} \cup T \subseteq V$ 
from $\{A\} \subseteq V$,
$T \subseteq V$, $\mathcal{P}^n(T) \subseteq \mathcal{P}^n(V) \subseteq U$ 
for all $n \in N_0$, using that $U$ is a subset-friendly set.
We conclude that $\mathcal{SP}(T)$ is the smallest subset-friendly set $U$
with $A \in U$.
\end{proof}
\begin{rem}\label{sf_remark}
In this section we have summarized basic properties of subset-friendly sets which serve as a guideline for Sections \ref{RST}, \ref{RST_Model}. 
\begin{itemize}
\item[(a)] Theorem \ref{spthm1} guarantees that for every set $A$
there is a subset-friendly set $U$ with $A \in U$.
This will be stated as a new set axiom in Section \ref{RST}.
\item[(b)]If $U$ is a subset-friendly set and $V \in U$, then
$\mathcal{P}(V) \in U$ and hence
$\mathcal{P}(V) \subseteq U$ from the transitivity of $U$. 
If moreover $A \subseteq V$, then $A \in \mathcal{P}(V)$
and $A \in U$ from $\mathcal{P}(V) \subseteq U$.
We see that $\{Y,Z\} \in U$ for the set $\{Y,Z\}$ in Condition 4
of Definition \ref{interpret_axioms}(b).
The transitivity of $V$ in Condition 4 guarantees 
in addition that $Y \cup Z \subseteq V$ and hence $Y \cup Z \in U$.
\end{itemize}
\end{rem}

\section{Formal mathematical systems}\label{FMS}

We use notations and results from \cite[Sections 3,4]{Ku}
and from Shoenfield's textbook \cite{Shoenfield}.
In \cite[Section 3]{Ku} a recursive system $S$ closely related 
to Smullyan's elementary formal systems in \cite{Sm} is embedded into a formal mathe\-matical system $M$. 
In \cite[(3.13)]{Ku} we use five rules of inference, namely rules (a)-(e).
Rule (e) enables formal induction with respect to the recursively enumerable relations generated by the underlying recursive system $S$.
For our application to axiomatic set theory in Section \ref{RST} we 
do not need the general syntax described in \cite[(3.1)-(3.15)]{Ku}
and impose three restrictions on our formal systems.

\textit{The first restriction.}
We put $S=S_{\emptyset}=[\,[\,];[\,];[\,]\,]$
in order to avoid the use of rule (e). 
Then we can shortly write $M=[A; P; B]$
instead of $M=[S_{\emptyset}; A; P ;B]$
for our formal systems. Here $A$ is the set of constants and function symbols, $P$ the set of predicate symbols 
and $B$ the set of basis axioms,
i.e. the axioms of the given theory.

\textit{The second restriction.}
To each predicate symbol $p \in P$ we assign a fixed arity $n \in \setN_0$
which will be given in the informal description of the formal system.

\textit{The third restriction.}
In \cite[(3.15)]{Ku} formal mathe\-matical systems $[M;\mathcal{L}]$
with restrictions in the argument lists 
of the formulas are introduced. The set of restricted argument lists 
$\mathcal{L}$ contains the variables and is closed with respect to substitutions. To each constant or function symbol $a \in A$ we assign a fixed arity $n \in \setN_0$. For $n=0$ we say that $a$ is a constant symbol, and for $n \geq 1$ we say that $a$ is an $n$-ary function symbol.
Then $\mathcal{L}$ consists only on terms which are generated by the following rules.
\begin{itemize}
\item[1.] We have $x \in \mathcal{L}$ 
for all variables $x \in X=\{\,{\bf x_1}\,,\,{\bf x_2}\,,\,{\bf x_3}\,,\, \ldots \,\}$.
\item[2.] We have $a \in \mathcal{L}$ for all constant symbols $a \in A$.
\item[3.] Let $n>0$ and let $a$ be an $n$-ary function symbol in $A$.\\
Then $a(\lambda_1 \ldots \lambda_n) \in \mathcal{L}$ for all terms
$\lambda_1,\ldots,\lambda_n \in \mathcal{L}$.
\end{itemize}
Let $\Pi(M;\mathcal{L})$ be the set of formulas provable
in $[M;\mathcal{L}]$ by using only the rules of inference (a),(b),(c),(d).
Under these restrictions we make use of the following definition 
and of three subsequent theorems.
\begin{defin}\label{extensions}
Given are two formal mathematical systems $[M;\mathcal{L}]$ and 
$[M';\mathcal{L}']$ with $M=[A; P; B]$ and $M'=[A'; P'; B']$.
\begin{itemize}
\item[(a)] We say that $[M';\mathcal{L}']$ is an \textit{extension} 
of $[M;\mathcal{L}]$ if
$$
A \subseteq A'\,, \quad 
P \subseteq P'\,, \quad 
\mathcal{L} \subseteq \mathcal{L}' ~\textrm{~and~} ~
\Pi(M;\mathcal{L}) \subseteq \Pi(M';\mathcal{L}') \,.
$$
\item[(b)] Let $[M';\mathcal{L}']$ be an extension of 
$[M;\mathcal{L}]$. If we have in addition
$$
F \in \Pi(M';\mathcal{L}') \implies F \in \Pi(M;\mathcal{L})
$$
for all formulas $F$ in $[M;\mathcal{L}]$, then 
$[M';\mathcal{L}']$ is called a \textit{conservative extension} of 
$[M;\mathcal{L}]$. 
\end{itemize}
\end{defin}

The proofs of the following three theorems are analogous
to that of Shoenfield's theorems on functional extensions 
and extensions by definitions in \cite[Section 4.5]{Shoenfield}
and \cite[Section 4.6]{Shoenfield}.

\begin{thm}\label{conserv0}
Let $[M;\mathcal{L}]$ be a formal mathe\-ma\-tical system.
We write $M=[A; P; B]$. 
We choose a new $n$-ary predicate symbol 
$p \notin P$ with $n \in \setN_0$ and 
form $P'=P \cup \{p\}$.
Let $x_1,\ldots,x_n$ be distinct variables,
and let $G$ be a formula in $[M;\mathcal{L}]$
in which no variable other than 
$x_1,\ldots,x_n$ is free.
We put 
$\begin{displaystyle}
B'= B \cup \Big\{ \leftrightarrow p \, x_1, \ldots ,x_n \, G \Big\}
\end{displaystyle}$ and
$M'=[A; P'; B']$
(for $n=0$ we have $p \, x_1, \ldots ,x_n = p$)\,.
Then $[M';\mathcal{L}]$ is a conservative extension of $[M;\mathcal{L}]$.
\end{thm}

\begin{thm}\label{conserv1}
Let $[M;\mathcal{L}]$ with $M=[A; P; B]$
be a formal mathe\-ma\-tical system.
Choose a new constant symbol $c \notin A$, form 
$\tilde{A}=A \cup \{c\}$ and
$\begin{displaystyle}
\tilde{\mathcal{L}}=\{\lambda\frac{c}{z}\,:\,
\lambda \in \mathcal{L}\textrm{~and~}~z \in X\}\,.
\end{displaystyle}$
Let $u,v \in X$ be distinct variables
and $G$ a formula in $[M;\mathcal{L}]$
with $\mbox{free}(G) \subseteq\{u\}$.
Assume that $v$ is not occurring bound in $G$.
\begin{itemize}
\item[(a)] 
We put $B'=B \cup \{ G\frac{c}{u}\}$, $M'=[\tilde{A}; P; B']$.
If the formula $\exists u G$ is provable in $[M;\mathcal{L}]$, then 
$[M';\tilde{\mathcal{L}}]$ is a conservative extension of $[M;\mathcal{L}]$.
\item[(b)] 
Put 
$\begin{displaystyle}
B''= B \cup \Big\{
\leftrightarrow \, \sim u,c ~ G
\Big\}
\end{displaystyle}$ and 
$M''=[\tilde{A}; P; B'']$.
If $\exists u G$ and
$\rightarrow G \rightarrow\,G\frac{v}{u} \sim u,v$ 
are both provable in $[M;\mathcal{L}]$, then 
$[M'';\tilde{\mathcal{L}}]$ is a conservative extension 
of $[M;\mathcal{L}]$.
\end{itemize}
\end{thm}

\begin{thm}\label{conserv2}
Let $[M;\mathcal{L}]$ with $M=[A; P; B]$ be a formal mathe\-ma\-tical system.
We choose a new n-ary function symbol $f \notin A$ with $n \in \setN$ 
and form $\tilde{A}=A \cup \{f\}$.
Let $\begin{displaystyle}
\tilde{\mathcal{L}}
\end{displaystyle}$
be the smallest set of terms satisfying
\begin{itemize}
\item[$\bullet$] $\begin{displaystyle}
\mathcal{L} \subseteq \tilde{\mathcal{L}}\,,
\end{displaystyle}$
\item[$\bullet$] $\begin{displaystyle}
f(y_1 \ldots y_n) \in \tilde{\mathcal{L}}
\end{displaystyle}$ ~\,
for all $y_1, \ldots ,y_n \in X$\,,
\item[$\bullet$] 
$
\lambda\frac{\mu}{z} \in \tilde{\mathcal{L}}
$
~\,
for all $z \in X$ and for all $\lambda,\mu \in \tilde{\mathcal{L}}$\,.
\end{itemize}
Let $u,v,x_1,\ldots,x_n$ be distinct variables,
and let $G$ be a formula  in $[M;\mathcal{L}]$ in which 
no variable other than $u,x_1,\ldots,x_n$ is free. 
Assume that $v,x_1,\ldots,x_n$ are not occurring bound in $G$.
\begin{itemize}
\item[(a)] Put 
$\begin{displaystyle}
B'= B \cup \Big\{ G \, \frac{f(x_1 \ldots x_n)}{u} \Big\}
\end{displaystyle}$ and 
$M'=[\tilde{A}; P; B']$.
If the formula $\exists u G$ is provable in $[M;\mathcal{L}]$, then 
$[M';\tilde{\mathcal{L}}]$ is a conservative extension of $[M;\mathcal{L}]$.
\item[(b)] 
Put 
$\begin{displaystyle}
B''= B \cup \Big\{
\leftrightarrow \, \sim u,f(x_1 \ldots x_n) \, G
\Big\}
\end{displaystyle}$ and 
$M''=[\tilde{A}; P; B'']$.
If $\exists u G$ and
$\rightarrow G \rightarrow\,G\frac{v}{u} \sim u,v$ 
are both provable in $[M;\mathcal{L}]$, then 
$[M'';\tilde{\mathcal{L}}]$ is a conservative extension of $[M;\mathcal{L}]$.
\end{itemize}
\end{thm}

\section{The system RST of reduced set theory}\label{RST}
Now we put $M^{(0)}=[\,[~]; [\in]; B^{(0)}]$ and $\mathcal{L}^{(0)}=X$ with the set of all variables given in \cite[(1.1)(c)]{Ku}. The formal set axioms in $B^{(0)}$ will be given below. In addition to \cite[(3.3)(a)]{Ku} we will only allow 
prime formulas $\sim r,s$ and $\in r,s$ with variables $r,s \in X$.
The 2-ary symbol $\in$ will be used in the formal system RST
as well as for the membership relation in our informal english text, which
will not lead to confusion.
First we define the formal mathematical system 
$\text{RST}=[M^{(0)};\mathcal{L}^{(0)}]$ 
with the following axioms for $B^{(0)}$, where
$t,u,v,w,x,y,z \in X$ are distinct variables 
which may vary and which may be chosen arbitrarily.
\begin{itemize}
\item[A1.] Axioms of extensionality.
$$\to ~ \forall y \, \leftrightarrow \, \in y,u \, \in y,v ~ \sim u,v$$
\item[A2.] Subset axioms.
$$\exists u \, \forall y \, \leftrightarrow \, \in y,u \, \& \in y,x \, F$$
with RST-formulas $F$ and $u,x \notin \text{var}(F)$\,.
\item[A3.] Axioms of regularity.
$$\rightarrow ~~ \exists y \, \in y,u ~~ 
\exists y \,\, \& \in y,u ~\, 
\neg \, \exists z \,  \& \, \in z,u \, \in z,y$$
\item[A4.] Axioms for pairing of sets. \quad 
$\exists u \,\, \& \in x,u \in y,u$
\item[A5.] Axioms for subset-friendly sets.
\begin{equation*}
\begin{split}
\exists u \, ~&~\& \, \& \, \& \, \in x,u \\ 
~&\quad \forall y \, \rightarrow \, \in y,u ~ 
\forall z \,  \rightarrow \, \in z,y \, \in z,u\\
~&\quad \forall y \, \rightarrow \, \in y,u ~
\exists z \, \&\,\in z,u\\
~&\quad\quad  \forall v ~  
\leftrightarrow ~ \in v,z ~ \forall w \, 
\rightarrow \, \in w,v \in w,y \\
~&\quad \forall y ~ \rightarrow ~ \in y,u
~\forall z ~ \rightarrow ~ \in z,u\\
~&\quad\quad  \exists v ~ \& ~ \& \in v,u ~ \& \in y,v \in z,v\\
~&\quad\quad\quad \forall w \rightarrow \, \in w,v~
~\forall t \rightarrow \, \in t,w \in t,v  
\end{split}
\end{equation*}
\item[A6.] Axioms of choice.
\begin{align*}
&\to \, \forall x \, \to\,\in x,u ~ \exists w \in w,x\\
&\to \, \forall x \forall y \, \to\,\in x,u\\
&\qquad \qquad \to\, \in y,u\\
&\qquad \qquad \to\, \exists w \,\&\,\in w,x\, \in w,y ~ \sim x,y\\
&\qquad \exists y \, \forall x \,\to\,\in x,u\\
&\qquad \quad \exists v \, \& \& \in v,x \in v,y\\
&\qquad \quad \quad \forall w \, \to \& \in w,x \in w,y ~ \sim v,w
\end{align*}
\end{itemize}
\begin{rem}\label{axiomsremark}
In RST we can apply rule (c) for the collision-free substi\-tution of free variables as well as the replacement of bound variables. Therefore it would be sufficient to choose a single set of distinct variables 
$t,u,v,w,x,y,z \in X$ for the formulation of the axioms. However, the more general choice of axioms is better suited for our purposes.\\
We use a Hilbert-style calculus for the formal mathematical systems 
$[\mbox{M};\mathcal{L}]$. Let $F$ be a formula in
$[\mbox{M};\mathcal{L}]$ and $x \in X$.
Then $F \in \Pi(\mbox{M};\mathcal{L})$ iff
$\forall x\, F \in \Pi(\mbox{M};\mathcal{L})$
from \cite[(3.11)(a),(3.13)(b)(d)]{Ku}. Hence we can use
formulas with free variables in our axioms.
Now $\Pi(RST)$ denotes the set of formulas provable in RST.
\end{rem}
\begin{defin}\label{sftheories}
Let $[M;\mathcal{L}]$ be an extension of RST 
(including RST itself). 
We say that $[M;\mathcal{L}]$ is a subset-friendly theory,
or sf-theory for short, if in addition the formulas
$\exists u \, \forall y \, \leftrightarrow \, \in y,u \, \& \in y,x \, F$
are provable in $[M;\mathcal{L}]$ for all collections of distinct variables $x,y,u$ and for all formulas $F$ in $[M;\mathcal{L}]$ 
with $u,x \notin \mbox{var}(F)$.
\end{defin}
\begin{lem}\label{lem1}
Let $[M;\mathcal{L}]$ be a sf-theory.
\begin{itemize}
\item[(a)] With distinct variables $y,u,v \in X$ the formula
$$\leftrightarrow ~ \forall y \, \leftrightarrow \, \in y,u \, \in y,v ~ \sim u,v$$
is provable in $[M;\mathcal{L}]$.
\item[(b)] With distinct variables $x,y,u \in X$ the formula
$$\exists u \, \forall y \, \leftrightarrow \, \in y,u \, 
\& \in y,x \, F$$
is provable in $[M;\mathcal{L}]$ 
for all $[M;\mathcal{L}]$-formulas $F$\\ 
with $u \notin \text{free}(F)$\,.
\end{itemize}
\end{lem}
\begin{proof}
(a) Since $[M;\mathcal{L}]$ is an extension of RST,
the exten\-sio\-nality \mbox{axioms} of RST 
are provable in $[M;\mathcal{L}]$. 
Now the statement simply follows from the  
axioms of equality in \cite[(3.10)]{Ku}.\\
(b) Let $F$ be a formula in $[M;\mathcal{L}]$ 
with $u \notin \text{free}(F)$\,.
Choose a variable $x' \notin \mbox{var}(F)$ different from $x,y,u$ and put 
$F' = F \frac{x'}{x}$.
We replace all the variables occurring bound in $F'$ by new variables,
different from $x,y,u$ and different from all variables in $F'$.
There results a new formula $F''$, and from \cite[(3.17)(b) Theorem]{Ku} 
we have
$\leftrightarrow F' F'' \in \Pi(M;\mathcal{L})$.
Due to \cite[(3.16) Lemma, part (a)]{Ku} we can apply rule (c)
to the last formula in order to replace the variable $x'$ by $x$.
Then we have
$\begin{displaystyle}
\leftrightarrow ~F'\frac{x}{x'} F''\frac{x}{x'} ~ 
= ~ \leftrightarrow ~F F''\frac{x}{x'}\in \Pi(M;\mathcal{L})
\end{displaystyle}$\,.
Now $u,x \notin \text{var}(F'')$, and 
$\begin{displaystyle}\exists u \, \forall y \, 
\leftrightarrow \, \in y,u \, \& \in y,x \, F''\end{displaystyle}$
is  provable in $[M;\mathcal{L}]$ 
due to Definition \ref{sftheories}. 
To this formula we can again apply rule (c)
in order to replace the variable $x'$ by $x$. Then
$$
\exists u \, \forall y \, \leftrightarrow 
\, \in y,u \, \& \in y,x \, F''\frac{x}{x'}
$$
and 
$\begin{displaystyle}
\leftrightarrow ~F F''\frac{x}{x'}
\end{displaystyle}$
are both provable in $[M;\mathcal{L}]$. 
We apply the equi\-valence theorem
\cite[(3.17)(a)]{Ku} and obtain that 
$\exists u \, \forall y \, \leftrightarrow \, \in y,u \, 
\& \in y,x \, F$
is provable in $[M;\mathcal{L}]$.
\end{proof}

\begin{lem}\label{lem2}
Let $F_1, F_2$ be formulas in
a formal mathematical system $[M;\mathcal{L}]$
and let $u,y$ be distinct variables. If
$\rightarrow F_1 \, F_2 \in \Pi(M;\mathcal{L})$ then
\begin{itemize}
\item[(a)] $\rightarrow \forall y F_1 ~\forall y F_2 ~
\in \Pi(M;\mathcal{L})$\,,
\item[(b)] $\rightarrow \exists y F_1 ~\exists y F_2 ~
\in \Pi(M;\mathcal{L})$\,,
\item[(c)] $\rightarrow \exists u \forall y F_1 ~\,
\exists u \forall y F_2 ~ \in \Pi(M;\mathcal{L})$\,.
\end{itemize}
\end{lem}
\begin{proof}
The following formulas are provable in $[M;\mathcal{L}]$.
\begin{itemize}
\item[1)] $\forall y \rightarrow F_1 \, F_2$ from 
$\rightarrow F_1 \, F_2 \in \Pi(M;\mathcal{L})$ and rule (d).
\item[2)] $\rightarrow ~ \forall y \rightarrow F_1 \, F_2~\,
\rightarrow \forall y F_1 ~\forall y F_2 ~$ from 
\cite[Theorem (3.18)(11)]{Ku}.
\item[3)] $\rightarrow \forall y F_1 ~\forall y F_2 ~$ from rule (b)
with 1) and 2), which shows (a).
\item[4)] $\rightarrow ~ \forall y \rightarrow F_1 \, F_2~\,
\rightarrow \exists y F_1 ~\exists y F_2 ~$ from 
\cite[Theorem (3.18)(12)]{Ku}.
\item[5)] $\rightarrow \exists y F_1 ~\exists y F_2 ~$ from rule (b)
with 1) and 4), which shows (b).
\item[6)] $\rightarrow \exists u \forall y F_1 ~
\exists u \forall y F_2 ~$, if we apply part (b) of the lemma on 3) .
\end{itemize}

\end{proof}

\begin{lem}\label{lem3}
Let $[M;{\mathcal L}]$ be a sf-theory,
$u$, $v$, $y$ be different \mbox{variables} and $F$ be a formula 
in $[M;{\mathcal L}]$. 
Then the following formula is provable in $[M;{\mathcal L}]$:
\begin{align*}
\to ~ & ~ \forall y \, \leftrightarrow \,~ \in y,u ~ F\\
\to ~ & ~ \forall y \, \leftrightarrow \,~ \in y,v ~ F\\
 & \sim u,v\,.
\end{align*}
\end{lem}
\begin{proof}
\noindent
\begin{itemize}
\item[1)] 
Due to the propositional calculus
\begin{align*}
\rightarrow ~~\,
\& ~ \leftrightarrow \,~ \in y,u ~ F
 \, \leftrightarrow \,~ \in y,v ~ F ~~
\leftrightarrow \,~ \in y,u \, \in y,v
\end{align*}
is provable in $[M;\mathcal{L}]$.
\item[2)] 
From 1) and Lemma \ref{lem2}(a) we obtain that
\begin{align*}
\rightarrow ~~\, \forall y ~
\& ~ \leftrightarrow \,~ \in y,u ~ F
 \, \leftrightarrow \,~ \in y,v ~ F ~~
\forall y \, \leftrightarrow \,~ \in y,u \, \in y,v
\end{align*}
is provable in $[M;\mathcal{L}]$.
\item[3)] Using Lemma \ref{lem1}(a) and the equivalence theorem 
\cite[(3.17)(a)]{Ku} we obtain from 2) that
\begin{align*}
\rightarrow ~~\, \forall y ~
\& ~ \leftrightarrow \,~ \in y,u ~ F
 \, \leftrightarrow \,~ \in y,v ~ F ~~
\sim u,v
\end{align*}
is provable in $[M;\mathcal{L}]$.
\item[4)] Due to 3) and \cite[(3.17)(a),(3.18)(15)]{Ku} the formula
\begin{align*}
\to ~\& & ~\forall y \, \leftrightarrow \,~ \in y,u ~ F\\
 &~ \forall y \, \leftrightarrow \,~ \in y,v ~ F\\
 &~ \sim u,v
\end{align*}
is provable in $[M;\mathcal{L}]$\,.
\end{itemize}
Now the statement of the lemma is a consequence 
of 4), using the propositional calculus.
\end{proof}

\begin{lem}\label{lem4}
Let $[M;{\mathcal L}]$ be a sf-theory,
$u$, $y$ be different \mbox{variables} and $F$ be a formula 
in $[M;{\mathcal L}]$  with $u \notin \mbox{free}(F)$. 
Then the following formula is provable in $[M;{\mathcal L}]$:
\begin{align*}
\leftrightarrow ~ & ~
\exists u \, \forall y \, \rightarrow \, F \, \in y,u\\
 & ~
\exists u \, \forall y \, \leftrightarrow \,~ \in y,u ~ F \,.
\end{align*}
\end{lem}
\begin{proof}
Let $v \notin \mbox{var}(F)$, $v$ different from $u, y$ be a new variable.
Then the following formulas are provable in $[M;{\mathcal L}]$:
\begin{itemize}
\item[1)] The axiom of the propositional calculus
\begin{align*}
\rightarrow ~ & ~\leftrightarrow ~\, \in y,u ~ \, \&\, \in y,v F\\
\rightarrow ~ & ~\rightarrow \, F \, \in y,v\\
 & ~\leftrightarrow \,~ \in y,u ~ F\,.
\end{align*}
\item[2)] From 1) and Lemma \ref{lem2}(a) 
\begin{align*}
\rightarrow ~ & \forall y ~\leftrightarrow ~\, \in y,u ~ \, \&\, \in y,v F\\
 & \forall y\, \rightarrow ~ \rightarrow \, F \, \in y,v\\
 & \qquad \quad \leftrightarrow \,~ \in y,u ~ F\,.
\end{align*}
\item[3)] From \cite[(3.18)(11)]{Ku} 
\begin{align*}
\rightarrow ~ & \forall y\, \rightarrow ~ \rightarrow \, F \, \in y,v\\
 & \leftrightarrow \,~ \in y,u ~ F\\
\rightarrow ~ & ~
\forall y \, \rightarrow \, F \, \in y,v\\
 & ~
\forall y \, \leftrightarrow \,~ \in y,u ~ F\,.
\end{align*}
\item[4)] From 2), 3) and the propositional calculus 
\begin{align*}
\rightarrow ~ & \forall y\, \leftrightarrow ~\in y,u ~ \& \, \in y,v\,F\\
\rightarrow ~ & ~ \forall y \, \rightarrow \, F \, \in y,v\\
 & ~\forall y \, \leftrightarrow \,~ \in y,u ~ F\,.
\end{align*}
\item[5)] From 4) and Lemma \ref{lem2}(b) 
\begin{align*}
\rightarrow ~ & ~ \exists u \forall y\, \leftrightarrow ~\in y,u ~ \& \, \in y,v\,F\\
&  ~ \exists u \rightarrow ~ \forall y \, \rightarrow \, F \, \in y,v\\
&  \qquad \quad \quad ~ \forall y \, \leftrightarrow \,~ \in y,u ~ F\,.
\end{align*}
\item[6)] From Lemma \ref{lem1}(b) due to $u \notin \mbox{free}(F)$
\begin{align*}
\exists u \forall y\, \leftrightarrow ~\in y,u ~ \& \, \in y,v\,F\,.\\
\end{align*}
\item[7)] From 5), 6) and rule (b) 
\begin{align*}
&  ~ \exists u \rightarrow ~ \forall y \, \rightarrow \, F \, \in y,v\\
&  \qquad \quad \quad ~ \forall y \, \leftrightarrow \,~ \in y,u ~ F\,.
\end{align*}
\item[8)] From \cite[(3.18)(21)]{Ku} and 7) with $J = \,\to$,
$u \notin \mbox{free}(F)$
\begin{align*}
&  \rightarrow ~ \forall y \, \rightarrow \, F \, \in y,v\\
&  \qquad \exists u ~ \forall y \, \leftrightarrow \,~ \in y,u ~ F\,.
\end{align*}
\item[9)] From 8) and rule (c) with $v \mapsto u$, using that $v \notin \mbox{var}(F)$:
\begin{align*}
&  \rightarrow ~ \forall y \, \rightarrow \, F \, \in y,u\\
&  \qquad \exists u ~ \forall y \, \leftrightarrow \,~ \in y,u ~ F\,.
\end{align*}
\item[10)] From 9) and rule (d) 
\begin{align*}
& \forall u \rightarrow ~ \forall y \, \rightarrow \, F \, \in y,u\\
&  ~\qquad \exists u ~ \forall y \, \leftrightarrow \,~ \in y,u ~ F\,.
\end{align*}
\item[11)] From 10) and \cite[(3.18)(18)]{Ku}
\begin{align*}
\rightarrow &~ \exists u \forall y \, \rightarrow \, F \, \in y,u\\
&~ \exists u ~ \forall y \, \leftrightarrow \,~ \in y,u ~ F\,.
\end{align*}
\item[12)] With $F_1 = \,\leftrightarrow \,~ \in y,u ~ F$ and  
$F_2 = \, \rightarrow \, F \, \in y,u$ from Lemma \ref{lem2}(c)
\begin{align*}
\rightarrow &~ \exists u ~ \forall y \, \leftrightarrow \,~ \in y,u ~ F\\
&~ \exists u \forall y \, \rightarrow \, F \, \in y,u\,.\\
\end{align*}
\end{itemize}
From 11) and 12) we obtain the desired result.
\end{proof}

\begin{lem}\label{const_erw}
Let $[M;\mathcal{L}]$ with $M=[A; P; B]$
be a sf-theory. Let $[M_c;\mathcal{L}_c]$ 
with $M_c=[A \cup \{c\}; P; B]$
result from $[M;\mathcal{L}]$
by adding a new constant symbol $c \notin A$.
Then $[M_c;\mathcal{L}_c]$ is a conservative 
and subset-friendly extension of $[M;\mathcal{L}]$.
\end{lem}
\begin{proof}
It follows from \cite[(4.9) Corollary]{Ku} that
$[M_c;\mathcal{L}_c]$ is a conservative extension of $[M;\mathcal{L}]$.
Let $x,y,u$ be distinct variables and $F$ be a formula
in $[M_c;\mathcal{L}_c]$ with $u,x \notin \mbox{var}(F)$.
Let $v \notin \mbox{var}(F) \cup \{x,y,u\}$ be a variable
and let $F'$ result from $F$ if we replace 
everywhere in $F$ the constant $c$ by $v$.
Then $F'$ is a formula in $[M;\mathcal{L}]$ with 
$u,x \notin \mbox{var}(F')$, and 
$$
H = \exists u \, \forall y \, \leftrightarrow \, \in y,u \, \& \in y,x \, F'
$$
is provable in the sf-theory $[M;\mathcal{L}]$.
We see that
$$
H\frac{c}{v} = \exists u \, \forall y \, 
\leftrightarrow \, \in y,u \, \& \in y,x \, F
$$
is provable in $[M_c;\mathcal{L}_c]$ from the substitution rule (c).
\end{proof}

\begin{thm}\label{durchschnitt}
Let $[M;\mathcal{L}]$ be a sf-theory
and let $x,y,z,v,w$ be distinct variables.
Then the formula
$$
\to ~ \exists v \in v,x ~ \exists w \, \forall z \,
\leftrightarrow \, \in z,w ~ \forall y \,\to \, \in y,x \in z,y
$$
is provable in $[M;\mathcal{L}]$.
\end{thm}
\begin{proof}
We add a new constant symbol $c$ to $[M;\mathcal{L}]$ and form
$[M_c;\mathcal{L}_c]$ as in Lemma \ref{const_erw}.
Let $[M';\mathcal{L}']$ result from $[M_c;\mathcal{L}_c]$ 
by adding the new basis axiom ~$\exists v \in v,c$~
to $[M_c;\mathcal{L}_c]$. We see from Lemma \ref{const_erw}
that $[M';\mathcal{L}']$ is a subset-friendly extension
of $[M;\mathcal{L}]$. Let $[M'';\mathcal{L}'']$  
result from $[M';\mathcal{L}']$ by adding a new constant symbol $d$ 
and the new basis axiom $\in d,c$ to $[M';\mathcal{L}']$.
Due to Lemma \ref{const_erw} and Theorem \ref{conserv1}(a)
we obtain that $[M'';\mathcal{L}'']$ is a conservative and
subset-friendly extension of $[M';\mathcal{L}']$. 
Due to Definition \ref{sftheories} the following
formula is provable in $[M'';\mathcal{L}'']$:
$$
\exists w \, \forall z \,
\leftrightarrow \, \in z,w ~ \& \in z,d ~
\forall y \,\to \, \in y,c \in z,y\,.
$$
Since ~$\in d,c$~is provable in $[M'';\mathcal{L}'']$,
we conclude that
$$
\leftrightarrow ~ 
\& \in z,d ~
\forall y \,\to \, \in y,c \in z,y ~\,
\forall y \,\to \, \in y,c \in z,y
$$
is provable in $[M'';\mathcal{L}'']$. 
We see that
$$
\exists w \, \forall z \,
\leftrightarrow \, \in z,w ~ 
\forall y \,\to \, \in y,c \in z,y
$$
is provable in $[M'';\mathcal{L}'']$
and hence in $[M';\mathcal{L}']$. 
Due to the deduction theorem \cite[(4.3)]{Ku} the formula
$$
\to ~ \exists v \in v,c ~ \exists w \, \forall z \,
\leftrightarrow \, \in z,w ~ \forall y \,\to \, \in y,c \in z,y
$$
is provable in $[M_c;\mathcal{L}_c]$.
Now \cite[(4.9) Corollary]{Ku} allows the generalization of
the constant $c$ in the last formula,
which concludes the proof of the theorem.
\end{proof}

Now we extend RST=$[M^{(0)};{\mathcal L}^{(0)}]$ by the following steps.

\begin{itemize}
\item[E1.] $[M^{(1)};{\mathcal L}^{(1)}]$ results from 
$[M^{(0)};{\mathcal L}^{(0)}]$ if we add the constant symbol 
$\emptyset$ and the following axioms to $[M^{(0)};{\mathcal L}^{(0)}]$:
$$
\leftrightarrow ~ \sim u,\emptyset ~ \, \forall y \, \neg \in y,u\,.
$$
Here $u,y \in X$ range over all pairs of distinct variables.
\item[E2.] $[M^{(2)};{\mathcal L}^{(2)}]$ results from 
$[M^{(1)};{\mathcal L}^{(1)}]$ with ${\mathcal L}^{(2)}={\mathcal L}^{(1)}$ if we add the 2-ary predicate symbol 
$\subseteq$ and the following axioms to $[M^{(1)};{\mathcal L}^{(1)}]$:
$$
\leftrightarrow ~ \subseteq u,v ~ \, \forall y \, \to \, \in y,u \in y,v\,.
$$
Here $u,v,y \in X$ range over all triples of distinct variables.
\item[E3.] $[M^{(3)};{\mathcal L}^{(3)}]$ results from 
$[M^{(2)};{\mathcal L}^{(2)}]$ if we add the 1-ary function symbol 
$\cup$ and the following axioms to $[M^{(2)};{\mathcal L}^{(2)}]$:
$$
\leftrightarrow ~ \sim u,\cup(x) ~ \, 
\forall z \,\leftrightarrow \, \in z,u ~ \exists y \, \& \in z,y \in y,x\
$$
for all quadruples $u,x,y,z \in X$ of distinct variables.
\item[E4.] $[M^{(4)};{\mathcal L}^{(4)}]$ results from 
$[M^{(3)};{\mathcal L}^{(3)}]$ if we add the 1-ary function symbol 
$\sigma$ and the following axioms to $[M^{(3)};{\mathcal L}^{(3)}]$:
$$
\leftrightarrow ~ \sim u,\sigma(x) ~ \, 
\forall y \, \leftrightarrow \, \in y,u ~ \sim y,x 
$$
for all triples $u,x,y\in X$ of distinct variables.\\
Now $\sigma(x)$ denotes the set $\{x\}$.
\item[E5.] $[M^{(5)};{\mathcal L}^{(5)}]$ results from 
$[M^{(4)};{\mathcal L}^{(4)}]$ if we add the 2-ary function symbol 
$\sigma_2$ and the following axioms to $[M^{(4)};{\mathcal L}^{(4)}]$:
$$
\leftrightarrow ~ \sim u,\sigma_2(x\,y) ~ \, \forall z\,
\leftrightarrow \, \in z,u ~ \vee \, \sim z,x \sim z,y\,.
$$
Here $u,x,y,z\in X$ range over all quadruples of distinct variables.
Now $\sigma_2(xy)$ denotes the set $\{x,y\}$.\\

\noindent
\textit{We introduce the following abbreviation, which is  
not part of the formal language: We define the successor
$x^+=\cup(\sigma_2(x \,\sigma(x)))$ of $x$.}
\item[E6.] $[M^{(6)};{\mathcal L}^{(6)}]$ results from 
$[M^{(5)};{\mathcal L}^{(5)}]$ if we add the 1-ary function symbol 
${\mathcal P}$ and the following axioms to $[M^{(5)};{\mathcal L}^{(5)}]$:
$$
\leftrightarrow ~ \sim z,{\mathcal P}(x) ~ \, 
\forall v \, \leftrightarrow \, \in v,z ~ 
\forall w\,\to\,\in w,v \in w,x \,.
$$
Here $v,w,x,z\in X$ range over all quadruples of distinct variables.
\end{itemize}
We finally obtain the new system 
$\mbox{RST}_{ext}=[\,M^{(6)};{\mathcal L}^{(6)}\,]$ with\\
$M^{(6)}=[\,A^{(6)};P^{(6)};B^{(6)}\,]$. The symbols are 
$
A^{(6)}=[\,\emptyset;\cup;\sigma;\sigma_2;
\mathcal{P}\,]$ and
$P^{(6)}=[\,\in;\subseteq\,]$\,.

The axioms of $B^{(6)}$ are given in A1-A6 and E1-E6, and 
${\mathcal L}^{(6)}$ is the set of terms constructed
from the constant $\emptyset$, the 1-ary function symbols
$\cup$, $\sigma$, $\mathcal{P}$
and the 2-ary function symbol $\sigma_2$.
\begin{rem}\label{sfaxiom}
Using the formal system $\mbox{RST}_{ext}$ we can rewrite
the axioms A5 for subset-friendly sets in the slightly simpler form
\begin{equation*}
\begin{split}
\exists u \, ~&~\& \, \& \, \& \, \in x,u \\ 
~&\quad \forall y \, \rightarrow \, \in y,u ~ 
\subseteq y,u\\
~&\quad \forall y \, \rightarrow \, \in y,u ~\in \mathcal{P}(y),u\\
~&\quad \forall y ~ \rightarrow ~ \in y,u
~\forall z ~ \rightarrow ~ \in z,u\\
~&\quad\quad  \exists v ~ \& ~ \& \in v,u ~ \subseteq \sigma_2(yz),v\\
~&\quad\quad\quad \forall w \rightarrow \, \in w,v~\subseteq w,v  \,.
\end{split}
\end{equation*}
\end{rem}
The purpose of the following theorem is twofold.
First it shows that every formula provable in $\mbox{RST}_{ext}$ 
can be replaced by an equivalent formula in RST 
which is already provable in RST.
Secondly it says that all axioms for sets in ZFC
given in \cite[Chapter 9]{Shoenfield} are already provable in RST, apart from the replacement axioms.
\begin{thm}\label{nice_axioms}
$\mbox{RST}_{ext}$ is a sf-theory and a conservative extension
of RST. The following formulas are provable 
in RST and more gene\-rally in every sf-theory $[M;{\mathcal L}]$ 
for all collections of distinct variables $x,y,z,u,v,w \in X$.
\begin{itemize}
\item[(a)] Existence of the empty set. 
$$\exists u \, \forall y \, \leftrightarrow \, \in y,u \, 
\& \in y,x \, \neg \in y,x \quad \mbox{and} \quad
\exists u \, \forall y \,  \neg \in y,u$$
\item[(b)] Existence of unions. $\exists u \, \forall z
\leftrightarrow \, \in z,u ~ \exists y \, \& \in z,y \in y,x$
\item[(c)] Existence of pair sets.
$$\exists u \, \forall z
\leftrightarrow \, \in z,u ~ \vee \, \sim z,x \sim z,y$$
\item[(d)] Existence of power sets.
$$\exists z \, \forall v ~  \leftrightarrow ~ \in v,z \,
\forall w \, \rightarrow \, \in w,v \in w,x$$
\item[(e)] Existence of an inductive set.
\begin{align*}
& \exists x \, \& \, \exists y ~ \&  \in y,x \, 
  \forall z \, \neg \in z,y\\
& \quad \forall y \to \, \in y,x~\exists z \, \& \, \in z,x\\
& \quad \quad
\forall v \, \leftrightarrow \, \in v,z \, \vee \in v,y \sim v,y	\\
\end{align*}
\end{itemize}
\end{thm}

\begin{proof}
To obtain the first part of the theorem 
we show for $j=1,\ldots,6$ that $[M^{(j)};{\mathcal L}^{(j)}]$ 
is a conservative extension of $[M^{(j-1)};{\mathcal L}^{(j-1)}]$.
For this purpose we make use of 
Theorems \ref{conserv0}, \ref{conserv1}(b) 
and \ref{conserv2}(b), which allows us
to replace formulas with new symbols step by step 
with equivalent formulas from the previous systems
by using the equivalence theorem from \cite[(3.17)(a)]{Ku},
the axioms of equality and the substitution rule (c).
We use \cite[(3.17)(b)]{Ku} for the replacement 
of bound variables to obtain the general formulation of
the axioms with diffe\-rent collections of variables.
Then we see that for $j=1,\ldots,6$ 
each conservative extension $[M^{(j)};{\mathcal L}^{(j)}]$ is a sf-theory.
We will see that the existence conditions
in Theorem \ref{conserv1}(b) and Theorem \ref{conserv2}(b) 
for the extensions $[M^{(j)};{\mathcal L}^{(j)}]$
with a new constant or function symbol 
are already provable in RST 
and hence in $[M^{(j-1)};{\mathcal L}^{(j-1)}]$.
In each case the corresponding uniqueness conditions 
will automatically result from Lemma \ref{lem3}.
Then the existence conditions (a)-(d) 
of the theorem are obtained as a by-product.
The proof of part (e) requires a little bit more effort.
\begin{itemize}
\item[E1.] Since RST=$[M^{(0)};{\mathcal L}^{(0)}]$ is a sf-theory,
we obtain from Lemma \ref{lem1}(b) that
$\exists u \, \forall y \, \leftrightarrow \, \in y,u \, 
\& \in y,x \, \neg \in y,x$
is provable in $[M^{(0)};{\mathcal L}^{(0)}]$.
Using that
$
\leftrightarrow \,
\leftrightarrow \, \in y,u \, 
\& \in y,x \, \neg \in y,x \,
\neg \in y,u
$
is an axiom of the propositional calculus, 
we obtain from the equivalence theorem \cite[(3.17)(a)]{Ku}
that the existence condition
$\exists u \, \forall y \,\neg \in y,u$ is provable in
$[M^{(0)};{\mathcal L}^{(0)}]$.
But the latter formula is equivalent to the first existence condition. 
We see that the two formulas in part (a) are provable in RST
and $\mbox{RST}_{ext}$.
\item[E2.] This extension has the desired properties due to
Theorem \ref{conserv0}.
\item[E3.] We obtain from A5 and Lemma \ref{lem2}(b) that
\begin{equation*}
\exists u  ~\& \, \in x,u 
~ \forall y \, \rightarrow \, \in y,u ~ 
\forall z \,  \rightarrow \, \in z,y \, \in z,u
\end{equation*}
is provable in RST.
We add a 1-ary function symbol $\lambda$ to RST and apply 
Theorem \ref{conserv2}(a) to the last formula. Now the following
formulas are provable in a conservative extension $\mbox{RST}'$ of RST:
\begin{itemize}
\item[1.] $\& \, \in x,\lambda(x) 
~ \forall y \, \rightarrow \, \in y,\lambda(x)  ~ 
\forall z \,  \rightarrow \, \in z,y \, \in z,\lambda(x)$
\item[2.] $\in x,\lambda(x)$ 
\item[3.] $\forall y \, \rightarrow \, \in y,\lambda(x)  ~ 
\forall z \,  \rightarrow \, \in z,y \, \in z,\lambda(x)$
\item[4.] $\rightarrow \, \in y,\lambda(x)  ~ 
\forall z \,  \rightarrow \, \in z,y \, \in z,\lambda(x)$
\item[5.] $\rightarrow \, \in x,\lambda(x)  ~ 
\forall z \,  \rightarrow \, \in z,x \, \in z,\lambda(x)$
\item[6.] $\forall z \,  \rightarrow \, \in z,x \, \in z,\lambda(x)$
\item[7.] $\rightarrow \, \in y,x \, \in y,\lambda(x)$
\item[8.] $\rightarrow \, \in y,\lambda(x)  ~ 
\rightarrow \, \in z,y \, \in z,\lambda(x)$ ~(4. and \cite[(3.16)(c)]{Ku})
\item[9.] $\rightarrow \, \& \in z,y \in y,x ~ 
\in z,\lambda(x)$ ~(7. and 8.)
\item[10.] $\forall z \,\forall y \,\rightarrow \, \& \in z,y \in y,x ~ 
\in z,\lambda(x)$ ~(9.)
\item[11.] $\forall z \,\rightarrow \, \exists y \,\& \in z,y \in y,x ~ 
\in z,\lambda(x)$ ~ (10. and \cite[(3.18)(18)]{Ku})
\item[12.] $\exists u \forall z \,\rightarrow \, 
\exists y \,\& \in z,y \in y,x ~ 
\in z,u$~ (11. and \cite[(3.19)]{Ku})
\end{itemize}
Since $\mbox{RST}'$ is a conservative extension of RST, the last formula 
is already provable in RST. Now we can apply Lemma \ref{lem4}
and obtain that the formula (b) of the theorem 
for the existence of unions is provable in RST.
\item[E4.] The following formulas are provable in RST:
\begin{itemize}
\item[1.] $\exists u \, \& \, \in x,u\, \in y,u$ ~ from A4
\item[2.] $\exists u \, \in x,u$ ~ (from 1.)
\item[3.] $\to \, \in x,u \, \to \, \sim y,x \in y,u$ 
\item[4.] $\forall y \, \to \, \in x,u \, \to \, \sim y,x \in y,u$ 
\item[5.] $\to \, \in x,u \, \forall y \, \to \, \sim y,x \in y,u$ ~
(4. and \cite[(3.18)(20)]{Ku})
\item[6.] $\to \exists u \, \in x,u ~
\exists u \forall y \to \, \sim y,x \in y,u$ ~ 
(5. and Lemma \ref{lem2}(b))
\item[7.] $\exists u \forall y \to \, \sim y,x \in y,u$ ~ 
(2. and 6.)
\item[8.] $\exists u \forall y \leftrightarrow \,\in y,u \, \sim y,x $ ~ 
(7. and Lemma \ref{lem4}).
\end{itemize}
\item[E5.] The following formulas are provable in RST:
\begin{itemize}
\item[1.] $\exists u \, \& \, \in x,u\, \in y,u$ ~ from A4
\item[2.] $\to \, \& \, \in x,u\, \in y,u~
 \to \, \vee \sim z,x \sim z,y \in z,u$
\item[3.] $\to \, \& \, \in x,u\, \in y,u~
\forall z \to \, \vee \sim z,x \sim z,y \in z,u$ \\
(2. and \cite[(3.18)(20)]{Ku})
\item[4.] $\to \exists u \, \& \, \in x,u\, \in y,u~
\exists u \forall z \to \, \vee \sim z,x \sim z,y \in z,u$\\
(3. and Lemma \ref{lem2}(b))
\item[5.] $\exists u \forall z \to \, \vee \sim z,x \sim z,y \in z,u$~
(1. and 4.)
\item[6.] $\exists u \forall z 
\leftrightarrow \,\in z,u \,\vee \sim z,x \sim z,y$ ~ 
(5. and Lemma \ref{lem4}).\\
The last formula is formula (c) in the theorem.
\end{itemize}
\item[ E6.]We obtain from A5 and Lemma \ref{lem2}(b) that
\begin{align*}
\exists u  ~\& \, \in x,u ~
\forall y \, \rightarrow \, \in y,u ~
\exists z \, \&\,\in z,u\\
\qquad \forall v ~  \leftrightarrow ~ \in v,z \,
\forall w \, \rightarrow \, \in w,v \in w,y
\end{align*}
is provable in RST.
We add a 1-ary function symbol $\mu$ to RST and apply 
Theorem \ref{conserv2}(a) to the last formula.
The following formulas are provable in a 
conservative extension $\mbox{RST}''$ of RST:
\begin{align*}
\& \, \in x,\mu(x) ~
\forall y \, \rightarrow \, \in y,\mu(x)  ~
\exists z \, \&\,\in z,\mu(x) \\
\qquad \forall v ~  \leftrightarrow ~ \in v,z \,
\forall w \, \rightarrow \, \in w,v \in w,y
\end{align*}
and from $\in x,\mu(x)$ the two formulas
\begin{align*}
\exists z \, \&\,\in z,\mu(x) 
~ \forall v ~  \leftrightarrow ~ \in v,z \,
\forall w \, \rightarrow \, \in w,v \in w,x\,,
\end{align*}
\begin{align*}
\exists z \, \forall v ~  \leftrightarrow ~ \in v,z \,
\forall w \, \rightarrow \, \in w,v \in w,x\,.
\end{align*}
\end{itemize}
The last formula is already provable in RST.
It is formula (d) in the theorem.
Hence it remains to show that formula (e) in the theorem
is provable in RST. For this purpose we use Remark \ref{sfaxiom}
and the formal system $\mbox{RST}_{ext}$.
Then we will not explicitely mention the use of axioms (E1)-(E6).
Let $[\mathcal{M}';\mathcal{L}']$ result from $\mbox{RST}_{ext}$
by adding a new constant symbol $c$.
Let $[\mathcal{M}'';\mathcal{L}'']$ 
result from $[\mathcal{M}';\mathcal{L}']$ 
by adding a new constant symbol $d$.
Finally, let $[\mathcal{M}''';\mathcal{L}''']$ 
result from $[\mathcal{M}'';\mathcal{L}'']$ 
by adding the new basis axiom $\& \& \& F_{1}''F_{2}''F_{3}''F_{4}''$,
where the formulas $F_{1}'', F_{2}'', F_{3}'', F_{4}''$
are given by the following abbreviations
for a collection $v,w,y,z \in X$ of distinct variables:
\begin{equation*}
\begin{split}
F_{1}'' =\, ~&~\in c,d \,,\\ 
F_{2}'' =~&\quad \forall y \, \rightarrow \, \in y,d ~ 
\subseteq y,d\,,\\
F_{3}'' =~&\quad \forall y \, \rightarrow \, \in y,d ~
\in \mathcal{P}(y),d\,,\\
F_{4}'' =~&\quad \forall y ~ \rightarrow ~ \in y,d
~\forall z ~ \rightarrow ~ \in z,d\\
~&\quad\quad  \exists v ~ \& ~ \& \in v,d ~ \subseteq \sigma_2(yz),v\\
~&\quad\quad\quad \forall w \rightarrow \, \in w,v~\subseteq w,v  \,.
\end{split}
\end{equation*}
Then the following formulas are provable 
in $[\mathcal{M}''';\mathcal{L}''']$:
\begin{itemize}
\item[S$_1$.] $\in c,d$\,, 
\item[S$_2$.] $\rightarrow \, \in y,d ~ 
\subseteq y,d$\,,
\item[S$_3$.]
$\rightarrow \, \in y,d 
~\in \mathcal{P}(y),d$\,,
\item[S$_4$.]
\begin{align*}
\rightarrow ~ \in y,d ~\rightarrow ~ \in z,d~
\exists v ~ \& ~ \& \in v,d \\ 
\subseteq \sigma_2(yz),v
~\forall w \rightarrow \, \in w,v~\subseteq w,v  \,.
\end{align*}
\end{itemize}
With given new distinct variables $t,y'$ the following formulas are also provable in $[\mathcal{M}''';\mathcal{L}''']$:
\begin{itemize}
\item[S$_5$.] $\to \, \subseteq y',y \,\in y',\mathcal{P}(y)$\,,
\item[S$_6$.] $\to \, \in \mathcal{P}(y),d \,\subseteq \mathcal{P}(y),d$
(from S$_2$)\,,
\item[S$_7$.] $\to \, \subseteq y',y \,\to \, \in y,d \in y',d$
(from S$_3$, S$_5$ and S$_6$)\,,
\item[S$_8$.] $\forall t \, \to \, \subseteq \sigma_2(yz),t \,\to \, 
\in t,d \in \sigma_2(yz),d$
(from S$_7$)\,,
\item[S$_9$.] $\rightarrow ~ \in y,d ~\rightarrow ~ 
\in z,d~\in \sigma_2(yz),d$
(from S$_4$ and S$_8$)\,,
\item[S$_{10}$.] 
\begin{align*}
\hspace{-0.3cm}
\rightarrow ~ \in y,d ~\rightarrow ~ \in z,d~
\exists v ~ \& ~ \& \in v,d \\ 
\subseteq \cup(\sigma_2(yz)),v
~\forall w \rightarrow \, \in w,v~\subseteq w,v  
\end{align*}
(from S$_4$).
\item[S$_{11}$.] $\forall t \, \to \, 
\subseteq \cup(\sigma_2(yz)),t \,\to \, 
\in t,d \in \cup(\sigma_2(yz)),d$
(from S$_7$)\,,
\item[S$_{12}$.] $\rightarrow ~ \in y,d ~\rightarrow ~ 
\in z,d~\in \cup(\sigma_2(yz)),d$
(from S$_{10}$ and S$_{11}$)\,,
\item[S$_{13}$.] $\forall y\, \rightarrow ~\in y,d ~
\forall z\, \rightarrow ~ \in z,d~\in \sigma_2(yz),d$
(from S$_9$)\,,
\item[S$_{14}$.] $\forall y\, \rightarrow ~\in y,d ~
\forall z\, \rightarrow ~ \in z,d~\in \cup(\sigma_2(yz)),d$
(from S$_{12}$)\,.
\item[S$_{15}$.] $\forall y'\, \forall y\,
\to \, \subseteq y',y \,\to \, \in y,d \in y',d$
(from S$_7$)\,.
\end{itemize}
For $j=1,\ldots,15$ we denote the formula in $S_j$ by $G_{j}''$, 
and we form $G_{j}'$ from $G_{j}''$ by replacing everywhere in
$G_{j}''$ the constant $d$ with a new variable $u$.
For $k=1,\ldots,4$ let $F_{k}'$ result from $F_{k}''$ 
if we replace everywhere in $F_{k}''$ the constant $d$ by $u$. 
We obtain from the deduction theorem \cite[(4.3)]{Ku}
that the formula
\begin{align*}
\to\, \& \& \& F_{1}''F_{2}''F_{3}''F_{4}''~
\& \&  G_{13}''G_{14}''G_{15}''
\end{align*}
is provable in $[\mathcal{M}'';\mathcal{L}'']$,
and from the generalization of the constant symbols $d$
with the variable $u$ that the formula
\begin{align*}
\to\, \& \& \& F_{1}'F_{2}'F_{3}'F_{4}'~
\& \&  G_{13}'G_{14}'G_{15}'
\end{align*}
is provable in $[\mathcal{M}';\mathcal{L}']$.
We form $G_{j}$ from $G_{j}'$ by replacing everywhere in
$G_{j}'$ the constant $c$ with a new variable $x$.
Let $F_{k}$ result from $F_{k}'$ if we replace everywhere in $F_{k}'$ the constant $c$ with $x$. Here it is only affecting $F_{1}'$.
From the generalization of the constant $c$ we see
that the formulas
\begin{align*}
\to\, \& \& \& F_{1}F_{2}F_{3}F_{4}~
\& \&  G_{13}G_{14}G_{15}
\end{align*}
and
\begin{align*}
\forall u\,\to\, \& \& \& F_{1}F_{2}F_{3}F_{4}~
\& \&  G_{13}G_{14}G_{15}
\end{align*}
are provable in $\mbox{RST}_{ext}$. We have
\begin{equation}\label{fabk}
\begin{split}
F_{1} =\, ~&~\in x,u \,,\\ 
F_{2} =~&\quad \forall y \, \rightarrow \, \in y,u ~ 
\subseteq y,u\,,\\
F_{3} =~&\quad \forall y \, \rightarrow \, \in y,u ~
\in \mathcal{P}(y),u\,,\\
F_{4} =~&\quad \forall y ~ \rightarrow ~ \in y,u
~\forall z ~ \rightarrow ~ \in z,u\\
~&\quad\quad  \exists v ~ \& ~ \& \in v,u ~ \subseteq \sigma_2(yz),v\\
~&\quad\quad\quad \forall w \rightarrow \, \in w,v~\subseteq w,v  
\end{split}
\end{equation}
and 
\begin{equation}\label{gabk}
\begin{split}
G_{13} = ~&\forall y\, \rightarrow ~\in y,u ~
\forall z\, \rightarrow ~ \in z,u~\in \sigma_2(yz),u\,,\\ 
G_{14} = ~&\forall y\, \rightarrow ~\in y,u ~
\forall z\, \rightarrow ~ \in z,u~\in \cup(\sigma_2(yz)),u\,,\\
G_{15} =~&\forall y'\, \forall y\,
\to \, \subseteq y',y \,\to \, \in y,u \in y',u\,.
\end{split}
\end{equation}
We see that $\exists \, u\,\& \& \& F_{1}F_{2}F_{3}F_{4}$
is the formula in Remark \ref{sfaxiom} and that
\begin{equation}\label{sf_folgerungen}
\exists \, u\,\& \& \& F_{1}F_{2}F_{3}F_{4} ~ \mbox{and}~
\forall u\,\to\, \& \& \& F_{1}F_{2}F_{3}F_{4}~
\& \&  G_{13}G_{14}G_{15}
\end{equation}
are both provable in $\mbox{RST}_{ext}$.
Let $F$ and $G$ be formulas in $\mbox{RST}_{ext}$,
and assume that $\exists u\,F$ as well as $\forall u\,\to\,FG$
are both provable in $\mbox{RST}_{ext}$. Then the following formulas are
provable in $\mbox{RST}_{ext}$ as well:
$\to\,FG$, $\to\,F \&FG$ and $\to\,\exists u\,F ~ \exists u\,\&FG$ 
from Lemma \ref{lem2}(b). We obtain that
\begin{equation}\label{stronger_exists}
\left\{\begin{split}
\exists u\,F \in \Pi(\mbox{RST}_{ext})\quad\mbox{and} \quad
\forall u\,\to\,FG \in \Pi(\mbox{RST}_{ext}) \\
\implies \exists u\,\&FG \in \Pi(\mbox{RST}_{ext})\,.\qquad\qquad
\end{split}
\right.
\end{equation}
We can apply \eqref{stronger_exists} to \eqref{sf_folgerungen}
to strengthen the existence condition for the formula
in Remark \ref{sfaxiom}.
If we put $x=\emptyset$ in the formulas \eqref{sf_folgerungen}, then
we obtain the existence of an inductive set
from the formal definition of the successor
$y^+=\cup(\sigma_2(y \,\sigma(y)))=\cup(\sigma_2(y \,\sigma_2(yy)))$.
\end{proof}
\begin{rem}\label{sfrich}
We can also introduce the following abbreviations, which are not part of the formal language: For $n \geq 2$ we put
$$\sigma(x_1 \ldots x_n)=
\cup(\sigma_2(\sigma(x_1 \ldots x_{n-1})\sigma(x_n)))\,.$$
Then $\sigma(x_1 \ldots x_n)$ denotes the set 
$\{x_1, \ldots ,x_n\}$ and $\cup(\sigma(x_1 \ldots x_n))$ the set
$x_1 \cup \ldots \cup x_n$ for all $n \in \setN$, respectively.
If we use $\langle x\,y \rangle$ 
as abbreviation for the ordered pair 
$\sigma_2(\sigma(x) \sigma_2(xy))=\sigma(\sigma(x) \sigma(xy))$, 
and more generally $\langle x_1 \ldots x_n \rangle = 
\langle \langle x_1 \ldots x_{n-1} \rangle x_n \rangle$ for $n \geq 3$,
then we can easily form cartesian product sets.

Any subset-friendly set $U$ satisfies the following closure properties:
\begin{itemize}
\item If $A \in U$, then $\cup(A) \in U$,
$\mathcal{TC}(A) \in U$ and $\mathcal{P}(A) \in U$,
\item If $A \subseteq V$ and $V \in U$, then $A \in U$,
\item If $A, B \in U$, then $A \cup B \in U$, $A \cap B \in U$
and $A \setminus B \in U$,
\item If $A_1, \ldots, A_n \in U$, then $\{ A_1, \ldots, A_n\} \in U$,
\item If $A_1, \ldots, A_n \in U$, 
then $A_1 \times \ldots \times A_n \in U$.
\end{itemize}
All these closure properties of the
subset-friendly sets $U$ are formally provable in RST.
This results from Theorem \ref{nice_axioms} and from
the prova\-bility of the formulas \eqref{sf_folgerungen} 
in $\mbox{RST}_{ext}$, see \eqref{fabk}, \eqref{gabk} 
and \eqref{stronger_exists} in the proof 
of Theorem \ref{nice_axioms}. We also see that Theorem \ref{tcthm2}(a)
can be formalized immediately in RST.
From Theorem \ref{durchschnitt} and Theorem \ref{nice_axioms}
we can also prove in RST that there is a smallest inductive set 
called $\omega$. Thus RST enables formal induction with respect to $\omega$
and the introduction of arithmetic operations for $\setN_0$,
$\setZ$, $\setQ$, $\setR$. We see that a considerable part of mathematics
can be formalized in RST.
\end{rem}

\section{Models for RST}\label{RST_Model}
To obtain models for RST we make free use 
of the intuitive notion of a set
as we did it in Section \ref{construct}.
We also accept the principles of regularity and choice
in the informal mathematical argumentation.
But all the set constructions we use can be
formalized in ZFC.
We recall \eqref{tcsp} and define
\begin{equation} \label{unisets}
\left\{ \begin{split}
	V_{0,0} & =\mathcal{SP}(\emptyset),&\\
	V_{n,k} &=\mathcal{SP}(V_{n,k-1}) ~&
	\textrm{for} ~ n \in \setN_0 ~\textrm{and} ~ k \in \setN,\\	
	\mathcal{U}_{n} &= \bigcup_{k=0}^{\infty} V_{n,k} ~&
	\textrm{for} ~ n \in \setN_0,\quad\quad\quad\quad\quad\\
	V_{n,0}&=\mathcal{SP}(\mathcal{U}_{n-1}) ~& 
	\textrm{for} ~ n \in \setN.\quad \quad\quad\quad\quad\\
	\end{split}
	\right.
\end{equation}
It follows from Theorem \ref{spthm1} by complete induction
that $V_{n,k}$ and $\mathcal{U}_{n}$ are subset-friendly sets
for all $n, k \in \setN_0$. Note that
$$
\mathcal{U}_{0} \in \mathcal{U}_{1} \in \mathcal{U}_{2}
\in \ldots \quad \mbox{and hence} \quad
\mathcal{U}_{0} \subset \mathcal{U}_{1} \subset \mathcal{U}_{2}
\subset \ldots
$$
from the transitivity of the sets $\mathcal{U}_{n}$
for $n \in \setN_0$ and the regularity principle.
\begin{thm}\label{thm_models}
For any fixed $n \in \setN_0$ the set $\mathcal{U}_{n}$
is the universe of a model for RST with the individuals
$A \in \mathcal{U}_{n}$  and with the true membership relation between these individuals. We call it the $\mathcal{U}_{n}$-model for short.
\end{thm}
\begin{proof}
The logical axioms \cite[(3.9),(3.10),(3.11)]{Ku} are generally valid
and rules \cite[(3.13)(a)(b)(c)(d)]{Ku} correspond to correct method 
of deduction. Hence it is sufficient to check that axioms (A1)-(A6)
are valid in the $\mathcal{U}_{n}$-model.
To each member $A \in \mathcal{U}_{n}$
we choose exactely one name $\alpha_A$, put
$W_n=\{\alpha_A \,:\, A \in \mathcal{U}_{n}\}$ and add all
these constant symbols to RST. We denote the resulting
formal mathematical system by RST$_n$.
We define $\mathcal{D}: W_n \to \mathcal{U}_{n}$ by
$\mathcal{D}(\alpha_A)=A$ and extend $\mathcal{D}$
in order to assign a truth value $\top$ or $\bot$
to all closed formulas of RST$_n$ as follows.
\begin{itemize}
\item For all sets $A, B \in \mathcal{U}_{n}$ we put 
$\mathcal{D}(\sim \alpha_A, \alpha_B)=\top$
iff $A=B$ as well as $\mathcal{D}(\in \alpha_A, \alpha_B)=\top$
iff $A \in B$.
\item $\mathcal{D}(\neg F)=\top$ iff $\mathcal{D}(F)=\bot$
for all closed formulas $F$ in RST$_n$.
\item $\mathcal{D}(\to F G)=\top$ iff 
$\mathcal{D}(F) \implies \mathcal{D}(G)$,
similarly for $\leftrightarrow$, $\&$ and $\vee$.
Here $F, G$ are any closed formulas in RST$_n$.
\item $\mathcal{D}(\forall x F)=\top$ iff 
$\mathcal{D}(F\frac{\alpha_A}{x}) = \top$ for all $A \in \mathcal{U}_{n}$.
Here $F$ is any formula in RST$_n$ with $\mbox{free}(F)\subseteq \{x\}$.
\item $\mathcal{D}(\exists x F)=\top$ 
iff there exists $A \in \mathcal{U}_{n}$ with
$\mathcal{D}(F\frac{\alpha_A}{x}) = \top$.
Here $F$ is any formula in RST$_n$ with $\mbox{free}(F)\subseteq \{x\}$.
\end{itemize}
\noindent
Let $F$ be a formula in RST$_n$ with $\mbox{free}(F)\subseteq \{x_1,\ldots,x_k\}$ and let
$k \in \setN$. Then we say that $F$ is valid in
$\mathcal{D}$ iff 
$\mathcal{D}(F\frac{\alpha_{A_1}}{x_1}
\ldots\frac{\alpha_{A_k}}{x_k})=\top$ 
for all $A_1,\ldots,A_k \in \mathcal{U}_n$,
see also Remark \ref{axiomsremark}.
Now we prove that $\mathcal{D}$ is the desired $\mathcal{U}_{n}$-model.
\begin{itemize}
\item[A1.] The extensionality axioms are valid in $\mathcal{D}$ 
since $\mathcal{U}_n$ is a transitive set: 
Let $U, V \in \mathcal{U}_{n}$ be any two sets\,.
Then
$$\mathcal{D}(\to ~ \forall y \, \leftrightarrow \, \in y,\alpha_U \, 
\in y,\alpha_V ~ \sim \alpha_U,\alpha_V)=\top$$
because we have for all sets $Y$ that
$Y \in U$ is equivalent to\\ 
$Y \in U \,\&\,Y \in \mathcal{U}_{n}$
and $Y \in V$ is equivalent 
to $Y \in V \,\&\, Y \in \mathcal{U}_{n}$.
\item[A2.] Let $u,x,y$ be distinct variables and $F$
be a formula in RST$_n$ with $u,x \notin \text{var}(F)$\,.
We replace all free variables in $F$ other than $y$
by arbitrary constants in $W_n$ and obtain a formula $G$ in RST$_n$ 
with $u,x \notin \text{var}(G)$ and $\mbox{free}(G)\subseteq \{y\}$\,.
We have to show for all $A \in \mathcal{U}_n$ that
$\mathcal{D}(\exists u \, \forall y \, 
\leftrightarrow \, \in y,u \, \& \in y,\alpha_A \, G)=\top\,.$
We define the set
$U = \{Y \in A\,:\,\mathcal{D}(G\frac{\alpha_Y}{y})=\top\,\}\,.$
From $U \subseteq A$ and $A \in \mathcal{U}_n$ 
with the subset-friendly set $\mathcal{U}_n$ we obtain that
$U \in \mathcal{U}_n$, see also Remark \ref{sf_remark}(b). 
Hence we can form the name $\alpha_U \in W_n$ and obtain that
$\mathcal{D}(\forall y \, 
\leftrightarrow \, \in y,\alpha_U \, \& \in y,\alpha_A \, G)=\top\,.$
Therefore the existence condition $\exists u \, \forall y \, 
\leftrightarrow \, \in y,u \, \& \in y,\alpha_A \, G$
is true in $\mathcal{D}$ as well.
\item[A3.] Here we prescribe any nonempty set $U \in \mathcal{U}_n$. 
Then we have $\mathcal{D}(\exists y \, \in y,\alpha_U)=\top$
from the transitivity of $\mathcal{U}_n$. 
From the regularity principle we have a set $Y$ with $Y \in U$
and $U \cap Y = \emptyset$. From $Y \in U \in \mathcal{U}_n$
with the transitive set $\mathcal{U}_n$ we have $\alpha_Y \in W_n$
and obtain that
$\mathcal{D}(\& \in \alpha_Y,\alpha_U ~\, 
\neg \, \exists z \,  \& \, \in z,\alpha_U \, \in z,\alpha_Y)=\top$\,. 
Now we see 
$$\qquad \mathcal{D}(\rightarrow \, \exists y \, \in y,\alpha_U \, 
\exists y \,\, \& \in y,\alpha_U ~ 
\neg \, \exists z \,  \& \, \in z,\alpha_U \, \in z,y)=\top\,.$$
\item[A4.] We have to show for all $A, Y \in \mathcal{U}_n$ that
$$\mathcal{D}(\exists u \,\, \& \in \alpha_A,u \in \alpha_Y,u)=\top\,.$$
We have indices $j, k \in \setN_0$ with
$A \in V_{n,j}$ and $Y \in V_{n,k}$. For $m = \max(j,k)$ we see
$A,Y \in V_{n,m} \subseteq \mathcal{U}_n$. From $U=V_{n,m}$ we can form the name $\alpha_U \in W_n$ of $V_{n,m}$ and conclude that
$$\mathcal{D}(\& \in \alpha_A,\alpha_U \in \alpha_Y,\alpha_U)=\top\,.$$
We see that the desired existence condition is true in $\mathcal{D}$ as well.
\item[A5.] For all $A \in \mathcal{U}_n$ we have 
$A \in V_{n,j} \in \mathcal{U}_n$ for some index $j \in \setN_0$ 
with the subset-friendly set $U=V_{n,j}$. This implies
\begin{equation*}
\mathcal{D}(\in \alpha_A,\alpha_U)=\top\,.
\end{equation*}
For all sets $Y \in U$ the conditions
$Z \in Y$ and $Z \in Y \cap \mathcal{U}_n$
are equivalent, and therefore
\begin{equation*}
\mathcal{D}(\forall y \, \rightarrow \, \in y,\alpha_U ~ 
\forall z \,  \rightarrow \, \in z,y \, \in z,\alpha_U)=\top\,.
\end{equation*}
Similarly we obtain from the properties 3. and 4.
in Definition \ref{interpret_axioms}(b) for the subset-friendly 
set $U \in \mathcal{U}_n$ that
\begin{equation*}
\begin{split}
~&\quad \mathcal{D}(\forall y \, \rightarrow \, \in y,\alpha_U ~
\exists z \, \&\,\in z,\alpha_U\\
~&\quad\quad  \forall v ~  
\leftrightarrow ~ \in v,z ~ \forall w \, 
\rightarrow \, \in w,v \in w,y)=\top\,, 
\end{split}
\end{equation*}
\begin{equation*}
\begin{split}
~&\quad \mathcal{D}(\forall y ~ \rightarrow ~ \in y,\alpha_U
~\forall z ~ \rightarrow ~ \in z,\alpha_U\\
~&\quad\quad  \exists v ~ \& ~ \& \in v,\alpha_U ~ \& \in y,v \in z,v\\
~&\quad\quad\quad \forall w \rightarrow \, \in w,v~
~\forall t \rightarrow \, \in t,w \in t,v )=\top\,. 
\end{split}
\end{equation*}
We see that axioms A5 are true in $\mathcal{D}$.
\item[A6.] Let $U \in \mathcal{U}_n$ be a set 
which has only nonempty and pairwise disjoint elements.
Using the transitivity of $\mathcal{U}_n$ we obtain from our assumptions
\begin{align*}
\mathcal{D}(\forall x \, \to\,\in x,\alpha_U ~ \exists w \in w,x)=\top
\end{align*}
as well as
\begin{align*}
&~ \, \mathcal{D}(\forall x \forall y \, \to\,\in x,\alpha_U ~ \to\, \in y,\alpha_U\\
&\qquad \qquad \to\, \exists w \,\&\,\in w,x\, \in w,y ~ \sim x,y)=\top\,.
\end{align*}

From the principle of choice we can find a set $Y'$ such that
$Y' \cap A$ has exactely one element $V(A) \in A$ for all $A \in U$.
Let $Y = \{V(A)\,:\,A \in U\}$. 
Then $Y \subseteq \cup(U) \in \mathcal{U}_n$ 
and hence $Y \in \mathcal{U}_n$ 
with $Y \cap A =\{V(A)\}$ for all $A \in U$.
We see that all sets involved other than $Y'$
are members of $\mathcal{U}_n$ and that
\begin{align*}
&\qquad \mathcal{D}(\exists y \, 
\forall x \,\to\,\in x,\alpha_U\\
&\qquad \qquad \quad \exists v \, \& \& \in v,x \in v,y\\
&\qquad \qquad \quad \quad \forall w \, \to \& \in w,x \in w,y 
~ \sim v,w)=\top\,.
\end{align*}
Therefore axioms A6 are true in $\mathcal{D}$.
\end{itemize}
\end{proof}

Due to Shoenfield \cite[Chapter 9.3]{Shoenfield} we say
that a set $\alpha$ is an ordinal if $\alpha$ is transitive and 
if every member of $\alpha$ is transitive.
We also mention the valuable lecture notes of Skolem \cite[Section 8]{Skolem}. 
There one can find the following facts about ordinals which we will use now:
\begin{itemize}
\item Members of ordinals are again ordinals.
\item If $\alpha$ is an ordinal then it is well-ordered by 
$\in$, i.e.\ for $\beta, \gamma \in \alpha$ with $\beta \neq \gamma$
we have either $\beta \in \gamma$ or $\gamma \in \beta$, and we do not have
an infinite sequence \eqref{forbidden_sequence} with members
$A_0, A_1, A_2,\ldots$ of $\alpha$.
\item Using transitivity we obtain for all \textit{ordinals} 
$\alpha$, $\beta$ that the conditions
$\beta \subseteq \alpha$ and 
[$\beta \in \alpha$ or $\beta =\alpha$]
are equivalent.
\end{itemize}

\begin{lem}\label{ordlem1}
Let $T$ be a transitive set and define its subset
$$
\gamma=\{\,\beta \in T\,:\,\beta ~\mbox{is~an ordinal}\,\}.
$$
Then $\gamma$ is an ordinal, and we have
$$
\gamma^+ = \gamma \cup \{\gamma\}
=\{\,\alpha \in \mathcal{P}(T)\,:\,\alpha ~\mbox{is~an ordinal}\,\}.
$$
\end{lem}
\begin{proof}
We define the set
$$
\gamma_*
=\{\,\alpha \in \mathcal{P}(T)\,:\,\alpha ~\mbox{is~an ordinal}\,\}
=\{\,\alpha \subseteq T\,:\,\alpha ~\mbox{is~an ordinal}\,\}.
$$
\begin{itemize}
\item[1.] Let
$\vartheta \in \cup(\gamma_*)$. Then we have an ordinal 
$\eta \in \gamma_*$ with $\vartheta \in \eta$ and 
$\vartheta \subseteq \eta \subseteq T$ 
from the transitivity of $\eta$ and from $\eta \in \gamma_*$.
Now $\vartheta$ is an ordinal with $\vartheta \subseteq T$, i.e.\
$\vartheta \in \gamma_*$. We have $\cup(\gamma_*) \subseteq \gamma_*$
and see that $\gamma_*$ is transitive and hence an ordinal.
\item[2.] Let
$\vartheta \in \cup(\gamma)$. Then we have an ordinal 
$\eta \in \gamma$ with 
$\vartheta \in \eta \in T$ and $\eta \subseteq T$
from the transitivity of $T$.
Now $\vartheta$ is an ordinal with $\vartheta \in T$, i.e.\
$\vartheta \in \gamma$. We have $\cup(\gamma) \subseteq \gamma$
and see that $\gamma$ is transitive and hence an ordinal.
\item[3.] $\gamma^+$ is also an ordinal, and we have
\begin{align*}
\gamma^+ &=\{\alpha \,:\,\alpha \in \gamma^+ \}\\
&=\{\alpha \,:\,\alpha \in \gamma \mbox{~or~} \alpha=\gamma\}
=\{\alpha \,:\,\alpha \subseteq \gamma\ ~\mbox{is~an ordinal}\}\,.
\end{align*}
The latter condition $\alpha \subseteq \gamma$ only holds 
for ordinals and hence is equivalent with $\alpha \subseteq T$
due to the definition of $\gamma$. 
We obtain $\gamma^+ =\gamma_*$.
\end{itemize}
\end{proof}

\begin{lem}\label{ordlem2}
Let $T$ be a transitive set. We define 
$$
\gamma=\{\,\beta \in T\,:\,\beta ~\mbox{is~an ordinal}\,\}
$$
and
$$
\tilde{\gamma} 
=\{\,\alpha \in \mathcal{SP}(T)\,:\,\alpha ~\mbox{is~an ordinal}\,\}.
$$
If $\gamma$ is at most countably infinite,
then $\tilde{\gamma}$ is a countably infinite set.
\end{lem}
\begin{proof}
We obtain from Theorem \ref{tcthm1} that
the sets $\mathcal{P}^n(T)$ with $n \in \setN_0$
form an increasing chain
$$
T = \mathcal{P}^0(T) \subseteq \mathcal{P}^1(T) 
\subseteq \mathcal{P}^2(T) \ldots
$$
of transitive sets with union $\mathcal{SP}(T)$.
It follows from Lemma \ref{ordlem1} and Theorem \ref{tcthm2}(a) 
that $\mathcal{P}^{n+1}(T) \cap \tilde{\gamma}$ has exactely
one ordinal more as a member than $\mathcal{P}^n(T) \cap \tilde{\gamma}$.
To conclude the proof we only have to note that
$$
\bigcup_{n=0}^{\infty} 
\left(\mathcal{P}^n(T) \cap \tilde{\gamma} \right)
=\tilde{\gamma}\,.
$$
\end{proof}

\begin{thm}\label{finalthm}
For $n \in \setN_0$ we recall the model set $\mathcal{U}_n$ 
in \eqref{unisets}. Then 
\begin{align}\label{ordun}
\{\,\alpha \in \mathcal{U}_n\,:\,\alpha ~\mbox{is~an ordinal}\,\}
\end{align}
is a countably infinite ordinal. 
\end{thm}
\begin{proof}
The set of all hereditarily finite sets is given by $V_{0,0}$ in \eqref{unisets}.
A countably infinite union of countably infinite sets 
is countable. Starting with $V_{0,0}$, 
it follows from Lemma \eqref{ordlem2} by complete induction
that the set \eqref{ordun} is only countably infinite.
Since $\mathcal{U}_n$ is transitive,
we see from Lemma \ref{ordlem1} that \eqref{ordun} is an ordinal.
\end{proof}

The RST-models given by \eqref{unisets} can only serve as an example.
Beside the ${\mathcal U}_n$-models there are various other 
models for RST with only countable ordinals. 
Note that there are uncountably many countable ordinals.
We use the notation $\alpha<\beta$
in order to indicate that $\alpha$ and $\beta$ are
ordinals with $\alpha \in \beta$, and we
put $\alpha+1=\alpha^+$,
$V_{\emptyset}=\emptyset$ as well as
$V_{\alpha+1}=\mathcal{P}(V_{\alpha})$,
$V_{\beta}=\bigcup_{\alpha<\beta} V_{\alpha}$
for all ordinals $\alpha$ and for all limit ordinals $\beta$.
Then we obtain from Lemma \ref{ordlem1} and Lemma \ref{ordlem2} 
with $\omega_1 = \{\alpha\,:\, \alpha \mbox{ is a countable ordinal}\,\}$
that
$
V_{\omega_1} 
$
is a model for RST which extends the $\mathcal{U}_n$-models
from Theorem \ref{thm_models}. We note that every ordinal 
$\alpha \in V_{\omega_1}$ is only countable.
Let $\varphi_{< \omega_1}$ be an RST-formula
which expresses the following statement:
``For every ordinal $\alpha$ and for every inductive set $I$
there exists an injective mapping $f : \alpha \to I$."
We add the statement $\varphi_{< \omega_1}$ to RST
and denote the resulting formal system by RST$_{<\omega_1}$.
Then $V_{\omega_1}$ still remains a model for RST$_{<\omega_1}$.
In view of Remark \ref{sfrich}, for most applications
of set theory the restriction to the model
set $V_{\omega_1}$ is sufficient. 
We can specify further properties of $V_{\omega_1}$
by adding step by step appropriate new symbols and axioms to RST
and RST$_{<\omega_1}$. 
To study extensions of RST is a quite natural approach since there is no such thing as complete axiomatics for set theory anyway.
For this reason we have presented the general frame of sf-theories
in Section \ref{RST}, which also includes the formal system 
ZFC$'$. Here ZFC$'$ results from RST 
by adding the replacement axioms to RST. 
We see from Theorem \ref{nice_axioms} and Theorem \ref{spthm1}
that ZFC and ZFC$'$ are equivalent formal systems.
Due to Shoenfield \cite[Chapter 9.1]{Shoenfield}
we have the following replacement axioms 
for ZFC and ZFC$'$ with given formulas $F$ in RST:
\begin{equation}\label{repax}
\left\{
\begin{split}
\to & \,\forall x\,\exists z\,\forall y\,\leftrightarrow \,F\,\in y,z\\
& \exists u\,\forall y\,\to \,\exists x ~ \& \in x,v\,F \in y,u\,.\\
\end{split}
\right.
\end{equation}
Here $x,y,z,u,v$ run through all collections of distinct variables
with the restriction that $u,v,z$ are not occurring in $F$.
The statement $\varphi_{< \omega_1}$ is false
in ZFC. But RST$_{<\omega_1}$ and appropriate extensions
of RST$_{<\omega_1}$ with model set $V_{\omega_1}$
still provide a valid set theory in its own right.

Next we will also derive a transitive countable model
for RST.

\begin{lem}\label{omegalem}
Let $(A_k)_{k \in \setN}$ be a sequence of sets 
with $A_k \in V_{\omega_1}$ for all $k \in \setN$.
Then we have
$$
\{ A_k\,:\,k \in \setN\} \in V_{\omega_1}\,.
$$
\end{lem}
\begin{proof}
We put $\mathcal{A}=\{ A_k\,:\,k \in \setN\}$. 
It follows from the definition of $V_{\omega_1}$ that 
for all $k \in \setN$ we have a countable ordinal $\alpha_k$ 
with $A_k \in V_{\alpha_k}$. Hence we can form the 
countable ordinal 
$\begin{displaystyle}
\alpha = \bigcup_{k \in \setN} \alpha_k \in \omega_1
\end{displaystyle}$ and see that
$
\mathcal{A} \subseteq V_{\alpha} \in V_{\omega_1}\,.
$
Since $\mathcal{A} \in \mathcal{P}(V_{\alpha})=V_{\alpha+1} \in V_{\omega_1}$, we obtain the desired result from the transitivity of $V_{\omega_1}$.
\end{proof}

\begin{thm}\label{finalthm2} 
We have a countable transitive set $\mathcal{U} \in V_{\omega_1}$ such
that the true binary membership relation 
$$
E=\{(A,B)\,:\, A \in B \mbox{~and~}
 A \in \mathcal{U} \mbox{~and~}  B \in \mathcal{U} \,\}
$$
makes $\mathcal{U}$ a model for RST 
which is elementarily equivalent to $V_{\omega_1}$.
\end{thm}
\begin{proof} The downward L\"owenheim-Skolem theorem 
gives an elementary equivalent countable submodel $\mathcal{U}'$ of 
$V_{\omega_1}$ for RST, 
see for example \cite[3.1 Theorems of Skolem]{H}.
The set $\mathcal{U}'$ is extensional in the sense of 
\cite[Definition 6.14]{J}, i.e. for any two distinct sets
$A, B \in \mathcal{U}'$ we have 
$A \cap \mathcal{U}' \neq B \cap \mathcal{U}'$.
Mostowski's Collapsing Theorem \cite[Theorem 6.15(ii)]{J}
gives a transitive set $\mathcal{U}$ and
an isomorphism $f : \mathcal{U}' \to \mathcal{U}$
such that $A \in B \Leftrightarrow f(A) \in f(B)$ 
for all $A, B \in \mathcal{U}'$. We see that $f$ preserves elementary equivalence. 

It remains to prove that $\mathcal{U} \in V_{\omega_1}$.
Due to Lemma \ref{omegalem} it is sufficient to show that
$\mathcal{U} \subseteq V_{\omega_1}$. 
We say that a set $A \in V_{\omega_1}$ has property
$\Phi(A)$ iff there holds the implication
$$
A \in \mathcal{U}' \Rightarrow f(A) \in V_{\omega_1}\,.
$$

For all $A \in \mathcal{U}'$ we have
$
f(A)=\{f(B) : B \in A \cap \mathcal{U}'\}\,.
$
Hence $\Phi(\emptyset)$ is valid from $\emptyset \in V_{\omega_1}$, 
and we obtain $\Phi(A)$ from Lemma \ref{omegalem}  
whenever there holds $\Phi(B)$ for every $B \in A$. 
We see by $\in$-induction from \cite[Theorem 6.4]{J} 
that $\Phi(A)$ is valid for all $A \in V_{\omega_1}$ and that
$\mathcal{U} \subseteq V_{\omega_1}$.
\end{proof}

\newpage

\begin{rem}\label{last_remark} We highlight two important aspects of our work.
\begin{itemize}
\item[(a)]
Axiomatic set theory cannot provide an absolutely valid description for large and especially uncountable ordinal numbers. By using the reduced set theory RST, we do not make existence statements for such ordinals. We do not see the resulting undecidable statements as a weakness of our theory, but rather as a strength. Indeed, they enable the expansion of the axiom system and the corresponding model sets and they force us to provide a more explicit description of the ordinals used in these models.
Now we can also study models of RST which are initial segments of the von Neumann hierarchy of sets.
\item[(b)] The subset-friendly theories in Definition \ref{sftheories} are extensions of RST 
with new axioms and new symbols in the formulas. These are well suited for the formulation of general axiomatic theories where set theo\-ry is available in the background. This is possible because in principle all mathematical objects can be reduced to sets, so we do not need a predicate symbol for sets in these theories. On the other hand, in a subset-friendly theory we do not necessarily have to describe \textit{how} we construct mathematical objects using sets. For example, we can set up a system of axioms for real analysis using set theory without having to describe the exact construction of the real numbers. 
\end{itemize}
\end{rem}

\section*{Acknowledgement}
I am very thankful to Gerald Warnecke for encouragement and support of my work.

\end{document}